\documentclass[a4paper, 10pt]{amsart}
\usepackage{amsmath}
\usepackage{amssymb}

\theoremstyle{plain}
\newtheorem{theorem}{\bf Theorem}[section]
\newtheorem{proposition}[theorem]{\bf Proposition}
\newtheorem{mproposition}[theorem]{\bf Main Proposition}
\newtheorem{lemma}[theorem]{\bf Lemma}
\newtheorem{corollary}[theorem]{\bf Corollary}
\newtheorem{conjecture}[theorem]{\bf Conjecture}

\theoremstyle{definition}

\newcommand{\vp}{\mathsf v}
\newcommand{\N}{\mathbb N}
\newcommand{\Z}{\mathbb Z}

\newcommand{\la}{\langle}
\newcommand{\ra}{\rangle}
\newcommand{\be}{\begin{equation}}
\newcommand{\ee}{\end{equation}}
\newcommand{\und}{\;\mbox{ and }\;}
\newcommand{\nn}{\nonumber}
\newcommand{\ber}{\begin{eqnarray}}
\newcommand{\eer}{\end{eqnarray}}

\newcommand{\Fc}{\mathcal F}

 \DeclareMathOperator{\ord}{ord}

 \DeclareMathOperator{\supp}{supp}

\newcommand{\Sum}[2]{\underset{#1}{\overset{#2}{\sum}}}

\renewcommand{\t}{\, | \,}

\newcommand{\Blfloor}{\Big\lfloor}
\newcommand{\Brfloor}{\Big\rfloor}

\begin{document}

\title{On  products of $k$ atoms II}

\address{Institute for Mathematics and Scientific Computing\\ University of Graz, NAWI Graz\\ Heinrichstra{\ss}e 36\\ 8010 Graz, Austria}
\email{alfred.geroldinger@uni-graz.at}

\address{Department of Mathematical Sciences \\ University of Memphis \\ Memphis, TN 38152, USA}
\email{diambri@hotmail.com}

\address{School of Mathematics \\ South China Normal University \\ Guangzhou 510631 \\ P.R. China}
\email{mcsypz@mail.sysu.edu.cn}

\author{Alfred Geroldinger and David J. Grynkiewicz and Pingzhi Yuan}

\thanks{This work was supported by the Austrian Science Fund FWF, Project No. P26036-N26.
Part of this manuscript was written while  the first author visited South China Normal University in Guangzhou. He wishes to thank the
School of Mathematics for their support and hospitality.}



\begin{abstract}
Let $H$ be a Krull monoid with  class group $G$ such that every class contains a prime divisor (for example,  rings of integers in algebraic number fields or holomorphy rings in algebraic function fields). For $k \in \mathbb N$, let $\mathcal U_k (H)$ denote the set of all $m \in \mathbb N$ with the following property: There exist atoms $u_1, \ldots, u_k, v_1, \ldots, v_m \in H$ such that $u_1 \cdot \ldots \cdot u_k = v_1 \cdot \ldots \cdot v_m$. Furthermore, let $\lambda_k (H) = \min \mathcal U_k (H)$  and $\rho_k (H) = \sup \mathcal U_k (H)$. The sets $\mathcal U_k (H) \subset \mathbb N$ are  intervals which are finite if and only if $G$ is finite.  Their minima $\lambda_k (H)$ can be expressed in terms of $\rho_k (H)$. The invariants $\rho_k (H)$ depend only on the class group $G$, and in the present paper they are studied with new methods from Additive Combinatorics.
\end{abstract}
\maketitle

\bigskip
\section{Introduction}
\bigskip

Let $H$ be a (commutative and cancelative) monoid. If an element $a \in H$ has a factorization $a = u_1 \cdot \ldots \cdot u_k$  into atoms $u_1, \ldots, u_k \in H$, then $k$ is called the length of the factorization, and the set $\mathsf L (a)$ of all possible lengths is called the set of lengths of $a$. For $k \in \mathbb N$, let $\mathcal U_k (H)$ denote the set of all $m \in \mathbb N$ with the following property: There exist atoms $u_1, \ldots, u_k, v_1, \ldots, v_m \in H$ such that $u_1 \cdot \ldots \cdot u_k = v_1 \cdot \ldots \cdot v_m$. Thus $\mathcal U_k (H)$ is the union of all sets of lengths containing $k$. Sets of lengths (and all invariants derived from them, such as  their unions) are the most investigated invariants in factorization theory. The sets $\mathcal U_k (H)$ were introduced by S.T. Chapman and W.W. Smith in  Dedekind domains (\cite{Ch-Sm90a})  and since then have been studied in settings ranging from numerical monoids over  Mori domains to monoids of modules (\cite{Fr-Ge08, Bl-Ga-Ge11a, B-C-H-P08, Ge-Ka-Re15a, Ba-Ge14b}).
Their suprema $\rho_k (H) = \sup \mathcal U_k (H)$  and their minima $\lambda_k (H) = \min \mathcal U_k (H)$ have received special attention. Indeed, the invariants $\rho_k (H)$ were first studied in the 1980s for rings of integers in algebraic number fields (\cite{Cz81, Sa-Za82}). The supremum over all $\rho_k (H)/k$ is called the elasticity of $H$, whose investigation was a key topic in early factorization theory (see \cite{An97a} for a survey, or to pick a few from many, see \cite[Problem 38]{Ch-Gl00a} and \cite{Ca-Ch95,A-C-C-S95,  Ka05a,Ch-Cl05, B-C-C-M09}).

In the present paper, we focus on Krull monoids having the property that every class in the class group contains a prime divisor. In Section \ref{2} we present the necessary background and Proposition \ref{2.4} gathers the present state of the art. Among others, if $H$ is such a Krull monoid with class group $G$ and $2 < |G| < \infty$, then
$\mathcal U_k (H) \subset \mathbb N$ is a finite interval, hence $\mathcal U_k (H) = [\lambda_k (H), \rho_k (H)]$, and  its minimum $\lambda_k (H)$ can be expressed in terms of $\rho_k (H)$. Moreover,  $\rho_k (H)$ depends only on the class group $G$ and hence it can be studied with methods from Additive Combinatorics. This is the starting point for the present paper. In Section \ref{3} we discuss open problems, formulate two conjectures (Conjecture \ref{3.3}),  and outline the program of the paper. The main results are Theorems \ref{4.1} and \ref{5.1}. The latter result is based on the recent characterization of all minimal zero-sum sequences of maximal length in groups of rank two (see Main Proposition \ref{5.4}).

\bigskip
\section{Unions of sets of lengths in Krull monoids: Background} \label{2}
\bigskip

Let $\mathbb N$ denote the set of positive integers and set $\mathbb N_0 = \mathbb N \cup
\{ 0 \}$. For  real numbers $a, b \in \mathbb R$, we denote by  $[a, b] = \{ x \in
\mathbb Z \mid a \le x \le b\}$ the discrete interval.   By a   monoid,  we  mean a commutative
semigroup with identity which satisfies the cancellation law (that
is, if $a,b ,c$ are elements of the monoid with $ab = ac$, then $b =
c$ follows). The multiplicative semigroup of non-zero elements of an
integral domain is a monoid.

Let $G$ be an  abelian group, and
let \ $A,\, B \subset G$ \ be subsets. Then $\langle A \rangle \subset G$ is the subgroup generated by $A$, $-A = \{-a \mid a \in A \}$, and $A + B = \{ a
+ b \mid a \in A, b \in B \}$ \ is the  sumset of $A$ and $B$. Furthermore,  $A$ is a generating set of $G$ if $\langle A \rangle = G$, and $A$ is a basis of $G$ if all elements of $A$ are nonzero and $G = \oplus_{a \in A} \langle a \rangle$.

\medskip
\noindent
{\bf Monoids and Sets of Lengths.}
A monoid $F$ is  free abelian, with basis $P \subset
F$ and we write $F = \mathcal F (P)$, if every $a \in F$ has a unique representation of the form
\[
a = \prod_{p \in P} p^{\mathsf v_p(a) } \quad \text{with} \quad
\mathsf v_p(a) \in \N_0 \ \text{ and } \ \mathsf v_p(a) = 0 \ \text{
for almost all } \ p \in P \,.
\]

Let $H$ be a monoid. We
denote by $H^{\times}$ the set of invertible elements of $H$ and by $\mathsf q (H)$ a quotient group of
$H$. For a subset $H_0 \subset H$, we denote by $[H_0] \subset H$
the submonoid generated by $H_0$. Let \ $a, b \in H$. We say that \
$a$  divides $b$ \ (and we write $a \t b$) if there is an
element $c \in H$ such that $b = ac$.
We denote by $\mathcal A (H)$  the   set of atoms (irreducible elements) of $H$. If $a=u_1 \cdot \ldots \cdot u_k$, where $k \in \N$ and $u_1, \ldots, u_k \in \mathcal A (H)$, then $k$ is called the length of the factorization and  $\mathsf L (a) = \{ k \in \N \mid a \ \text{has a factorization of length } \ k \} \subset \N$ is the set of lengths of $a$. For convenience, we set $\mathsf L (a) = \{0\}$ if $a \in H^{\times}$.  Furthermore, we denote by
\[
\mathcal L (H) = \{\mathsf L (a) \mid a \in H \} \quad \text{the  system of sets of lengths of} \quad H \,.
\]
Next we define the central concept of this paper. Let $k \in \mathbb N$ and suppose that  $H \ne H^{\times}$. Then
\[
\mathcal U_k (H) \ = \ \bigcup_{a \in H, \, k \in \mathsf L (a)}  \mathsf L (a)
\]
is the union of all sets of lengths containing $k$. Thus, $\mathcal U_k (H)$ is the set of all $m \in \N$ such that there are atoms $u_1, \ldots, u_k, v_1, \ldots, v_m$ with $u_1 \cdot \ldots \cdot u_k = v_1 \cdot \ldots \cdot v_m$.
Finally, we define
\[
\rho_k (H) = \sup \, \mathcal U_k (H) \quad \text{  and } \quad \lambda_k (H)= \min \ \mathcal U_k (H) \,.
\]

\medskip
\noindent
{\bf Krull monoids.} A monoid homomorphism  $\varphi \colon H \to F$ is said to be a divisor homomorphism if $\varphi (a) \t \varphi (b)$ in $F$ implies that $a \t b$ in $H$ for all $a, b \in H$. A monoid $H$ is said to be a Krull monoid if one of the following equivalent properties is satisfied (see \cite[Theorem 2.4.8]{Ge-HK06a} or \cite{HK98}):
\begin{enumerate}
\item[(a)] $H$ is completely integrally closed and satisfies the ascending chain condition on divisorial ideals.

\item[(b)] $H$ has a divisor homomorphism into a free abelian monoid.

\item[(c)] $H$ has a divisor theory: this is a divisor homomorphism $\varphi \colon H \to F = \mathcal F (P)$ into a free abelian monoid such that for each $p \in P$ there is a finite set $E \subset H$ with $p = \gcd \big( \varphi (E) \big)$.
\end{enumerate}

Let $H$ be a Krull monoid. Then every non-unit has a factorization into atoms, and all sets of lengths are finite. A divisor theory $\varphi \colon H \to F=\mathcal F (P)$ is essentially unique, and the class group
$\mathcal C (H) = \mathsf q (F)/\mathsf q (\varphi (H))$  depends only on $H$. It will be written additively, and we say that every class contains a prime divisor if, for every $g \in \mathcal C (H)$,
there is a $p \in P$ with $p \in g$.

An integral domain $R$ is a Krull domain if and only if its
multiplicative monoid $R \setminus \{0\}$ is a Krull monoid, and  Property (a) shows that a
noetherian domain is Krull if and only if it is integrally closed.
Rings of integers, holomorphy rings in algebraic function fields, and
regular congruence monoids in these domains are Krull monoids with finite
class group such that every class contains a prime divisor
(\cite[Section 2.11]{Ge-HK06a}). Monoid domains and power series
domains that are Krull are discussed in \cite{Gi84, Ki01, Ki-Pa01}. For monoids of modules which are Krull we refer the reader to \cite{Ba-Wi13a, Ba-Ge14b,Fa06a}.

Main portions of the arithmetic of a Krull monoid---in particular, all questions dealing with sets of lengths---can be studied in  the monoid of zero-sum sequences over its class group. We  provide the relevant concepts and summarize the connection in the next subsection.

\medskip
\noindent {\bf Transfer homomorphisms and Zero-sum sequences.} Let $G$ be an additively written abelian group,   $G_0 \subset G$   a subset, and let $\mathcal F (G_0)$ be the free abelian monoid with basis $G_0$. According to the tradition of combinatorial number theory, the elements of $\mathcal F(G_0)$ are called \ {\it sequences} over \
$G_0$. If $S = g_1 \cdot \ldots \cdot g_l$, where $l \in \N_0$ and $g_1, \ldots , g_l \in G_0$, then $\sigma (S) = g_1+ \ldots + g_l$ is called the sum of $S$, and the monoid
\[
\mathcal B(G_0) = \{ S  \in \mathcal F(G_0) \mid \sigma (S) =0\} \subset \mathcal F (G_0)
\]
is called the \ {\it monoid of zero-sum sequences} \ over \ $G_0$.
Since the embedding $\mathcal B (G_0) \hookrightarrow \mathcal F (G_0)$ is a divisor homomorphism, Property (b) shows that $\mathcal B (G_0)$ is a Krull monoid. The monoid $\mathcal B (G)$ is factorial if and only if $|G|\le 2$. If $|G| \ge 3$, then $\mathcal B (G)$ is a Krull monoid with class group isomorphic to $G$ and every class contains precisely one prime divisor.

For every arithmetical invariant \ $*(H)$ \ defined for a monoid
$H$, it is usual to  write $*(G)$ instead of $*(\mathcal B(G))$ (although this is an abuse of language,
but there will be no danger of confusion). In
particular, we set \ $\mathcal A (G) = \mathcal A (\mathcal B (G))$, $\mathcal L (G) = \mathcal L (\mathcal B (G))$, $\mathcal U_k (G) = \mathcal U_k (\mathcal B (G))$, $\rho_k (G) = \rho_k (\mathcal B (G))$, and $\lambda_k (G) = \lambda_k ( \mathcal B (G))$.

\smallskip
The next two propositions reveal the universal role of monoids of zero-sum sequences.

\smallskip
\begin{proposition} \label{2.1}
Let $H$ be a Krull monoid with class group $G$ such that every class contains a prime divisor. Then there is a transfer homomorphism $\boldsymbol \beta \colon   H \to \mathcal B (G)$.
In particular, for every $k \in \N$, we have
\[
\mathcal U_k (H) = \mathcal U_k  (G)  \,, \ \lambda_k (H) = \lambda_k (G)  \,,   \quad \text{and} \quad \rho_k (H) =  \rho_k  (G)   \,.
\]
\end{proposition}

\begin{proof}
See \cite[Theorem 3.4.10]{Ge-HK06a}.
\end{proof}

Whereas the proof of the above result is quite straightforward,
there are recent deep results showing that there are non-Krull monoids (even non-commutative rings) which allow  transfer homomorphisms to monoids of zero-sum sequences.

\medskip
\begin{proposition} \label{2.2}~

\begin{enumerate}
\item Let $\mathcal O$ be a holomorphy ring in a global field $K$, $A$ a central simple algebra over $K$, and $H$ a classical maximal $\mathcal O$-order of $A$ such that every stably free left $R$-ideal is free. Then $\mathcal U_k (H) = \mathcal U_k (G)$ for every $k \in \N$, where  $G$ is a ray class group of $\mathcal O$ and hence finite abelian.

\smallskip
\item Let $H$ be a seminormal order in a holomorphy ring of a global field with principal order $\widehat H$ such that the natural map $\mathfrak X (\widehat H) \to \mathfrak X (H)$ is bijective and there is an isomorphism $\overline{\vartheta}\colon \mathcal{C}_v(H)\rightarrow \mathcal{C}_v(\widehat{H})$ between the $v$-class groups. Then $\mathcal U_k (H) = \mathcal U_k (G)$ for every $k \in N$, where $G = \mathcal{C}_v(H)$ is  finite abelian.
\end{enumerate}
\end{proposition}

\begin{proof}
1. See \cite[Theorem 1.1]{Sm13a}, and \cite{Ba-Sm15} for related results of this flavor.

2. See \cite[Theorem 5.8]{Ge-Ka-Re15a} for a more general result in the setting of weakly Krull monoids.
\end{proof}

\medskip
We need some more notation for sequences over abelian groups (it is consistent with \cite{Ge-HK06a, Ge-Ru09, Gr13a}). As before, we fix an additive abelian group $G$ and a subset $G_0 \subset G$. Let
\[
S = g_1 \cdot \ldots \cdot g_l = \prod_{g \in G_0} g^{\mathsf v_g
(S)} \in \mathcal F (G_0) \,,
\]
be a sequence over $G_0$ (whenever we write a sequence in this way, we tacitly assume that $l \in \N_0$ and $g_1, \ldots, g_l \in G_0$). We set $-S = (-g_1) \cdot \ldots \cdot (-g_l)$ and $\mathsf v_{G_1}(S) = \sum_{g \in G_1} \mathsf v_{g}(S)$ for a subset $G_1 \subset G_0$. We call $\mathsf v_g(S)$ the {\it multiplicity} of $g$ in $S$,
\[
\begin{aligned}
|S| & = l = \sum_{g \in G} \mathsf v_g (S) \in \mathbb N_0 \
\text{the \ {\it length} \ of \ $S$} \,,  \quad \supp (S) = \{g \in
G \mid \mathsf v_g (S) > 0 \} \subset G \ \text{the \ {\it support}
\ of \
$S$} \,, \\
\sigma (S) & = \sum_{i = 1}^l g_i \ \text{the \ {\it sum} \ of \
$S$} \,, \quad  \text{and} \quad \Sigma (S) = \Big\{ \sum_{i \in I} g_i
\mid \emptyset \ne I \subset [1,l] \Big\} \ \text{ the \ {\it set of
subsums} \ of \ $S$} \,.
\end{aligned}
\]
For a sequence $T\in \mathcal F(G_0)$, we write $\gcd(S,T)\in \Fc(G_0)$ for the maximal length subsequence dividing $S$ and $T$. We write $T\mid S$ to indicate that $T$ is a subsequence of $S$, in which case $ST^{-1} = T^{-1}S$ denotes the subsequence obtained from $S$ by removing the terms from $T$.
The sequence $S$ is said to be
\begin{itemize}
\item {\it zero-sum free} \ if \ $0 \notin \Sigma (S)$,

\item a {\it zero-sum sequence} \ if \ $\sigma (S) = 0$,

\item a {\it minimal zero-sum sequence}  \ if it is a nontrivial zero-sum sequence and every proper  subsequence is zero-sum free.
\end{itemize}
Clearly,  the minimal zero-sum sequences are precisely the atoms of the monoid $\mathcal B (G_0)$, and they play a central role in our investigations. Now suppose that $G$ is finite. For $n \in \mathbb N$, let $C_n$ denote a cyclic group with $n$
elements. If  $|G| > 1$, then we have
\[
G \cong C_{n_1} \oplus \ldots \oplus C_{n_r}  \,, \quad \text{and we
set} \quad \mathsf d^* (G) = \sum_{i=1}^r (n_i-1) \quad\und\quad\mathsf D^*(G)=\mathsf d^*(G)+1\,,
\]
where $r = \mathsf r (G) \in \mathbb N$ is the \ {\it rank} of
$G$, $n_1, \ldots , n_r \in \mathbb N$ are integers with  $1 < n_1
\mid \ldots \mid n_r$ and $n_r = \exp (G)$ is the exponent of $G$.
If $|G| = 1$, then $\mathsf r (G) = 0$, $\exp (G) = 1$, and $\mathsf d^* (G) = 0$.
The {\it Davenport constant} $\mathsf D (G)$  of $G$ is the maximal length of a minimal zero-sum sequence over $G$, thus
\[
\mathsf  D (G) = \max \bigl\{ |U| \, \bigm| \; U \in \mathcal A
(G) \bigr\} \in \N \,.
\]
(note that $\mathcal A (G)$ is finite). In other words, $\mathsf D (G)$ is the smallest integer $\ell$ such that every sequence $S$ over $G$ has a nontrivial zero-sum subsequence. We denote by $\mathsf d (G)$ the maximal length of a zero-sum free sequence, and clearly we have $1+\mathsf d (G)=\mathsf D (G)$. The next proposition gathers some facts on the Davenport constant which we will use without further mention.

\begin{proposition} \label{2.3}
Let $G$ be a finite abelian group.
\begin{enumerate}
\item $\mathsf D^* (G) \le \mathsf D (G) \le |G|$.

\smallskip
\item If $G$ is a $p$-group or $\mathsf r (G) \le 2$, then $\mathsf D^* (G) = \mathsf D (G)$.

\smallskip
\item $\mathsf D (G)=1$ if and only if $|G|=1$, $\mathsf D (G)=2$ if and only if $|G|=2$, and $\mathsf D (G)=3$ if and only if $G$ is cyclic of order $|G|=3$ or isomorphic to $C_2 \oplus C_2$.
\end{enumerate}
\end{proposition}

\begin{proof}
See \cite[Chapter 5]{Ge-HK06a}. Note that 1. is elementary and that 3. is a simple consequence of 1. and 2. There are more groups $G$ with $\mathsf D^* (G)=\mathsf D (G)$ (beyond the ones listed in 2.), but we do not have equality in general  (\cite{Ge-Li-Ph12, Sm10a}).
\end{proof}

\smallskip
The next proposition gathers the state of the art on unions of sets of lengths.

\medskip
\begin{proposition} \label{2.4}
Let  $H$ be a Krull monoid with  class group  $G$  such that every class contains a prime divisor.
\begin{enumerate}
\item If \ $|G|\le 2$, then $\mathcal U_k (H)= \{k\}$ for all $k \in \N$.

\smallskip
\item If  \ $2 < |G| < \infty$, then, for all $k \in \N$, we have $\mathcal U_k (H) = [\lambda_k (G), \rho_k (G)]$ and
      \[
      \lambda_{k \mathsf D (G) + j} (H) = \begin{cases}
                                          2k \quad  & \text{for} \quad j = 0 \\
                                          2k+1 \quad  & \text{for} \quad j \in [1, \rho_{2k+1}(G) - k \mathsf D (G)]  \\
                                          2k+2 \quad  & \text{for} \quad j \in [\rho_{2k+1}(G) - k \mathsf D (G)+1, \mathsf D (G) - 1]  \,,
                                          \end{cases}
\]
provided that $k \mathsf D (G) + j \ge  1$.

\smallskip
\item If \ $G$ is infinite, then $\mathcal U_k (H) = \N_{\ge 2}$ for all $k \ge 2$.
\end{enumerate}
\end{proposition}

\begin{proof}
1. is classical, for 2. see \cite[Theorem 4.1]{Fr-Ge08} and \cite[Section 3.1]{Ge09a}, and 3. follows from \cite[Theorem 7.4.1]{Ge-HK06a}.
\end{proof}

\medskip
Let $H$ be a Krull monoid with  class group $G$ such that every class contains a prime divisor, or any of the monoids in Proposition \ref{2.2}.
Then Propositions \ref{2.1}, \ref{2.2},  and \ref{2.4}  show that, for a complete description of the sets $\mathcal U_k (H)$ of $H$, it remains to study
the invariants $\rho_{k}  (G )$ of an associated monoid of zero-sum sequences. This is the goal of the present paper.

\bigskip
\section{The extremal cases in the crucial inequality} \label{3}
\bigskip

We start with a simple and well-known lemma. For convenience, we provide its short proof.

\medskip
\begin{lemma} \label{3.1}
Let $G$ be a finite abelian group with $|G| \ge 3$, and let $k, l \in \mathbb N$.
\begin{enumerate}
\item $k + l \le \rho_k (G) + \rho_l (G) \le \rho_{k+l} (G)$.

\smallskip
\item $\rho_{2k} (G) = k \mathsf D (G)$ \ and
      \begin{equation} \label{crucialinequality}
      1 + k \mathsf D (G) \le \rho_{2k+1} (G) \le k \mathsf D (G) + \Blfloor  \frac{\mathsf D (G)}{2} \Brfloor \,.
      \end{equation}
      In particular, if $\mathsf D (G)=3$, then $\rho_{2k+1}(G)=k \mathsf D (G)+1$.

\smallskip
\item If $\rho_{2k+1} (G) \ge  m$ for some $m \in \N$ and $l \ge k$, then $\rho_{2l+1} (G) \geq m +(l-k)\mathsf D(G)$.
\end{enumerate}
\end{lemma}

\begin{proof}
1. By definition, we have $\rho_m (G) \ge m$ for each $m \in \N$ and hence $k + l \le \rho_k (G) + \rho_l (G)$. Since $\mathcal U_k (G) + \mathcal U_l (G) \subset \mathcal U_{k+l}(G)$, it follows that
\[
\rho_k (G) + \rho_l (G) = \sup \mathcal U_k (G) + \sup \mathcal U_l (G) \le \sup \mathcal U_{k+l}(G) =  \rho_{k+l} (G) \,.
\]

\smallskip
2. A simple counting argument shows that $\rho_k (G) \le k \mathsf D (G)/2$; furthermore, if $U = g_1 \cdot \ldots \cdot g_{\mathsf D (G)} \in \mathcal A (G)$, then $(-U)^kU^k= \prod_{i=1}^{\mathsf D (G)} \big( (-g_i)g_i \big)^k$, whence $k \mathsf D (G) \le \rho_{2k}(G)$ and thus $\rho_{2k}(G) = k \mathsf D (G)$ (details can be found in \cite[Theorem 2.3.1]{Ge09a}). Using this and 1., we infer that
\[
1+ \mathsf k \mathsf D (G) = \rho_1(G) + \rho_{2k}(G) \le \rho_{2k+1}(G) \le \frac{(2k+1)\mathsf D (G)}{2} \,.
\]
Clearly, $\mathsf D (G)=3$ implies that equality holds in both inequalities above.

\smallskip
3. By 1. and 2., it follows that
\[
\rho_{2l+1}(G) \ge \rho_{2k+1}(G) + \rho_{2(l-k)}(G) \ge m +(l-k)\mathsf D(G) \,. \qedhere
\]
\end{proof}

\smallskip
Our starting point is the crucial inequality (\ref{crucialinequality}). We conjecture that cyclic groups are the only groups where equality holds on the left hand side, whereas, for all noncyclic groups, there is a $k^* \in \N$ such that equality holds on the right hand side for all $k \ge k^*$. We are going to outline this in greater detail (see Conjecture \ref{3.3} and Corollary \ref{3.4}).

\medskip
\begin{proposition} \label{3.2}
Let $G$ be a finite abelian group with $\mathsf D (G) \ge 4$.
\begin{enumerate}
\item If there exist $U \in \mathcal A (G)$ and $S_1, S_2 \in \mathcal F (G)$ such that
                    \[
                    U = S_1S_2 \,, \ |U| = \mathsf D (G) \quad \text{and} \quad \Sigma (S_1) \cup \Sigma (-S_2) \ne G \setminus \{0\} \,,
                    \]
      then $\rho_3 (G) > \mathsf D (G)+1$.

\smallskip
\item If $G$ is cyclic, then the property in 1. does not hold and $\rho_{2k+1} (G) =k \mathsf D (G)+1$ for each $k \in \N$.

\smallskip
\item The following conditions are equivalent{\rm \,:}
      \begin{enumerate}
      \item There is a $k^* \in \N$ such that $\rho_{2k^*+1} (G) = k^* \mathsf D (G) + \Blfloor \frac{\mathsf  D (G)}{2} \Brfloor$.

      \item There is a $k^* \in \N$ such that
            \[
            \rho_{2k+1} (G) = k \mathsf D (G) + \Blfloor \frac{\mathsf  D (G)}{2} \Brfloor \qquad \text{for every} \ k \ge k^* \,.
            \]
      \end{enumerate}
\end{enumerate}
\end{proposition}

\begin{proof}
1. Let  $S_1, S_2$, and $U$ have the above property. Then we choose an element
\[
g \in G \setminus \bigl(  \Sigma (S_1) \cup \Sigma (-S_2) \cup \{0\}    \bigr) \,.
\]
Since $g \notin  \Sigma (S_1)$, the sequence $(-S_1)g$ is zero-sum free and $U_2 = (-S_1)g \bigl(\sigma (S_1)-g \bigr) \in \mathcal A (G)$.
Similarly, it follows that $U_3 = (-S_2)(-g) \bigl( \sigma (S_2)+g \bigr) \in \mathcal A (G)$. Since $-\bigl(\sigma (S_1)-g \bigr) = \sigma (S_2)+g$ in view of $0=\sigma(U)=\sigma(S_1)+\sigma(S_2)$,
the product $U U_2U_3$ has a factorization into $|S_1| + |S_2| + 2 = \mathsf D (G) + 2$ atoms of length $2$.

\smallskip
2. Suppose that $G$ is cyclic of order $|G|=n$. Then \cite[Theorem 5.3]{Ga-Ge09b}  implies that $\rho_{2k+1} (G)=k\mathsf D (G)+1$ for each $k \in \N$ (see \cite[Theorem 5.3.1]{Ge09a}
 for a slightly modified proof). Clearly, every $U \in \mathcal A (G)$ of length $|U|=|G|$ has the form $U = g^n$ for some $g \in G$ with $\ord (g)=n$. Thus there are no $S_1$ and  $S_2$ with the given properties.

\smallskip
3. (a)\, $\Rightarrow$\, (b) \ If $l \in \N$, then  Lemma \ref{3.1} implies that
\[
\begin{aligned}
(k^*+l)\mathsf D (G) + \Blfloor \frac{\mathsf  D (G)}{2} \Brfloor & \ge \rho_{2(k^*+l)+1} (G)  \ge \rho_{2k^*+1}(G) + \rho_{2l}(G) \\
 & = \left(  k^* \mathsf D (G)  + \Blfloor \frac{\mathsf  D (G)}{2} \Brfloor \right) + l \mathsf D (G) = (k^*+l)\mathsf D (G) + \Blfloor \frac{\mathsf  D (G)}{2} \Brfloor \,.
\end{aligned}
\]

(b)\, $\Rightarrow$\, (a) \ Obvious.
\end{proof}

\medskip
\begin{conjecture} \label{3.3}
Let $G$ be a  noncyclic finite abelian group  with $\mathsf D (G) \ge 4$. Then the following two conditions hold{\rm \,:}
\begin{enumerate}
\item[{\bf C1.}] There exist $U \in \mathcal A (G)$ and $S_1, S_2 \in \mathcal F (G)$ such that
                    \[
                    U = S_1S_2 \,, \ |U| = \mathsf D (G) \,, \quad \text{and} \quad \Sigma (S_1) \cup \Sigma (-S_2) \ne G \setminus \{0\} \,.
                    \]

\smallskip
\item[{\bf C2.}] There exists some $k^* \in \N$ such that
            \[
            \rho_{2k+1} (G) = k \mathsf D (G) + \Blfloor \frac{\mathsf  D (G)}{2} \Brfloor \qquad \text{for each } \quad k \ge k^* \,.
            \]
\end{enumerate}
\end{conjecture}

\medskip
In Proposition \ref{3.5},  we show that Conjecture {\bf C1} holds for groups $G$ with $\mathsf D (G)=\mathsf D^* (G)$. All results of the present paper support Conjecture {\bf C2}. In particular, Theorem \ref{4.1}  provides groups satisfying {\bf C2} with $k^* = 1$, and  Theorem \ref{5.1} shows that  {\bf C2} need not hold with $k^* =1$.

We start with some consequences of the above conjecture.
The Characterization Problem is a central topic in factorization theory for Krull monoids
(we refer to \cite[Sections 7.1 - 7.3]{Ge-HK06a} for general information, and to \cite{Sc09b, Sc09c, B-G-G-P13a, Ge-Sc16a} for recent progress). The Characterization Problem studies the question  whether or not the system of sets of lengths of a Krull monoid, which has a prime divisor in every class, determines the class group. Thus, if $G$ and  $G'$ are two finite abelian groups with $\mathsf D (G) \ge 4$ such that $\mathcal L (G) = \mathcal L (G')$, does it follow that $G$ and $G'$ are isomorphic?  The answer is affirmative (among others) for groups of rank at most two, and there are no counter examples so far.  Corollary \ref{3.4} offers a simple proof in case of cyclic groups  which relies only on the $\rho_k (\cdot)$-invariants.

\medskip
\begin{corollary} \label{3.4}
Suppose that Conjecture {\bf C1} holds.
\begin{enumerate}
\smallskip
\item Let $H$ be a Krull monoid with finite class group $G$ such that every class contains a prime divisor and suppose that $\mathsf D (G) \ge 4$. Then the following statements are equivalent\,{\rm :}
\begin{enumerate}
\smallskip
\item[(a)] $G$ is cyclic.

\smallskip
\item[(b)] $\rho_{2k+1} (H) = k \mathsf D (G) + 1$ for every $k \in \mathbb N$.

\smallskip
\item[(c)] $\rho_3 (H) = \mathsf D (G)+1$.
\end{enumerate}

\smallskip
\item Let $G$ be cyclic with $\mathsf D (G) \ge 4$. If $G'$ is a finite abelian group with $\mathcal L (G) = \mathcal L (G')$, then $G \cong G'$.
\end{enumerate}
\end{corollary}

\begin{proof}
1. By Proposition \ref{2.1}, it suffices to consider $\rho_k (G)$ for all $k \in \N$.
The implication \
(a) \,$\Rightarrow$\, (b) \ follows from Proposition \ref{3.2}.2, and
(b) \,$\Rightarrow$\, (c) \ is obvious.

\smallskip
(c) \,$\Rightarrow$\, (a) \ If $G$ would be  noncyclic, then {\bf C1} and Proposition \ref{3.2}.1 would imply that $\rho_3 (G) > \mathsf D (G)+1$.

\smallskip
2.
Suppose that $\mathcal L (G) = \mathcal L (G')$. Then
\[
\mathsf D (G) = \rho_2 (G) = \rho_2 (G') = \mathsf D (G') \quad \text{and} \quad \mathsf D (G')+1 = \mathsf D (G)+1 = \rho_3 (G) = \rho_3 (G') \,.
\]
Thus 1. implies that $G'$ is cyclic, and since $|G|=\mathsf D (G) = \mathsf D (G')=|G'|$,  $G$ and $G'$ are isomorphic.
\end{proof}

\smallskip
For Conjecture {\bf C1} and for Corollary \ref{3.4}, the assumption $\mathsf D (G)\ge 4$ is crucial. By Proposition \ref{2.3}, the groups $C_3$ and $C_2 \oplus C_2$ are the only groups (up to isomorphism) whose Davenport constant is equal to three. The group $C_2 \oplus C_2$ does not satisfy {\bf C1}, $\rho_3 (C_2 \oplus C_2) = 4$ (in contrast to Corollary \ref{3.4}.1), and $\mathcal L (C_3)=\mathcal L (C_2 \oplus C_2)$ (see \cite[Theorem 7.3.2]{Ge-HK06a}).

\smallskip
The only groups $G$ with $\mathsf D (G)>\mathsf D^* (G)$, for which the precise value of $\mathsf D (G)$ is known, are groups of the form $C_2^4 \oplus C_{2n}$. We verify Conjecture {\bf C1} for them too.

\medskip
\begin{proposition}  \label{3.5}
Let $G$ be a noncyclic finite abelian group with $\mathsf D (G) \ge 4$.
\begin{enumerate}
\item Let  $G = C_{n_1} \oplus \ldots \oplus C_{n_r}$  where $1 < n_1 \t \ldots \t  n_r$ and suppose that
there is some $s \in [1, r-1]$ such that $n_s < n_{s+1} = \ldots = n_r$. Then $\rho_3(G)\geq \mathsf D^* (G)  + n_s$.

\smallskip
\item If  $\mathsf D (G) = \mathsf D^* (G)$, then Conjecture {\bf C1} holds.

\smallskip
\item If $G=C_2^4 \oplus C_{2n}$ with $n \ge 70$, then Conjecture {\bf C1} holds.
\end{enumerate}
\end{proposition}

\begin{proof}
Let  $\{e_1, \ldots , e_r\}$ be a basis of $G$ with $\ord (e_i) = n_i$ for $i
\in [1, r]$ and $n_1\mid\ldots\mid n_r$. Set $e_0 = e_1+ \ldots + e_{r-1}$.

\smallskip
1. Let
\[
\begin{aligned}
U_1 & = e_1^{n_1-1} \cdot \ldots \cdot e_r^{n_r-1} (e_0+e_r), \\
U_2 & = (-e_1)^{n_1-1} \cdot \ldots \cdot (-e_{s-1})^{n_{s-1}-1}(-e_s+e_r)^{n_s-1} (-e_{s+1})^{n_{s+1}-1}
          \cdot \ldots \cdot (-e_r)^{n_r-1}(-e_0 -n_se_r), \\
U_3 & = (-e_s)^{n_s-1}(e_s-e_r)^{n_s-1}(-e_0-e_r)(e_0+n_s e_r).
\end{aligned}
\]
Then the $U_i$ are each atoms,
and clearly $U_1U_2U_3$ is a product of
\[
\frac{1}{2} |U_1U_2U_3| = \frac{1}{2} ( 2 \mathsf D^* (G) + 2n_s) =
\mathsf D^* (G) + n_s
\]
atoms of length $2$. The assertion follows.

\smallskip
2. We  consider the sequence
\[
U = e_1^{n_1-1} \cdot \ldots \cdot e_r^{n_r-1}e_0 \,,
\]
and distinguish two cases.

First, suppose that $n_r > 2$. We set $S_1=e_r^{n_r-1}$ and $S_2=S_1^{-1}U$. Then $-e_1- \ldots - e_{r-1}+e_r \notin \Sigma (S_1)$ and $e_1+ \ldots + e_{r-1}-e_r \notin \Sigma (S_2)$ because $-e_r \ne e_r$.

Second, suppose that $n_r=2$. Then $G$ is an elementary $2$-group and $r \ge 3$ (as $\mathsf d(G)\geq 3$). We set $S_1=e_1e_2$ and $S_2=e_3 \cdot \ldots \cdot e_re_0$. Then $e_2+e_3 \notin \Sigma (S_1) \cup \Sigma (-S_2)$.

\smallskip
3. Suppose that $\ord (e_1)= \ldots = \ord (e_4)=2$ and $\ord (e_5)=2n$ with $n \ge 70$. If $n$ is even, then $\mathsf D (G)=\mathsf D^* (G)$ by \cite[Theorem 5.8]{Sa-Ch14a}, and the assertion follows from 2. Suppose that $n$ is odd. Then $\mathsf D (G)=\mathsf D^* (G)+1$ by \cite[Theorem 5.8]{Sa-Ch14a}. By \cite[Theorem 4]{Ge-Sc92}, the sequence
\[
U = (e_1+e_5)(e_2+e_5)(e_3+e_5)(e_4+e_5)(e_0-e_1)(e_0-e_2)(e_0-e_3)(e_0-e_4+e_5)^{2n-3}(-e_5)
\]
is a minimal zero-sum sequence of length $|U|=\mathsf D (G)$. We set
\[
S_1 = (e_0-e_4+e_5)^{2n-3}(-e_5) \quad\und \quad S_2=S_1^{-1}U .
\]
Then the element $e_1+e_2+e_3-2e_5 \notin \Sigma (S_1)$, and we assert that its inverse---namely $e_1+e_2+e_3+2e_5=e_0-e_4+e_5$---does not lie in $\Sigma (S_2)$. If there would be a subsequence $T$ of $S_2$ with $\sigma (T)=e_1+e_2+e_3+2e_5$, then we would have $|T|=2$. But none of the subsequences of $S_2$ of length two has sum $e_1+e_2+e_3+2e_5$, a contradiction.
\end{proof}

\bigskip
\section{Inductive Bounds} \label{4}
\bigskip

It is the aim of this section to prove the following result which confirms Conjecture {\bf C2} (with $k^*=1$) for the groups $G$ having the  form below and  satisfying $\mathsf D (G) = \mathsf D^* (G)$.

\medskip
\begin{theorem} \label{4.1}
Let $H$ be a Krull monoid with finite class group $G$ such that every class contains a prime divisor. Suppose that  $G=C_{n_1}^{s_1} \oplus \ldots \oplus C_{n_r}^{s_r}$ where $1 < n_1\mid\ldots\mid n_r$ and  $s_i\geq 2$ for all $i \in [1,r]$. Then
\[
\rho_{2k+1} (H) \geq  \mathsf D^* (G) + \Blfloor \frac{\mathsf  D^* (G)}{2}\Brfloor+(k-1)\mathsf D(G)  \quad \mbox{ for every $k\geq 1$} \,.
\]
In particular, if $\mathsf D(G)=\mathsf D^*(G)$, then $\rho_{2k+1} (H) = k \mathsf D (G) + \Blfloor \frac{\mathsf  D (G)}{2} \Brfloor$ for every $k\geq 1$.
\end{theorem}

\smallskip
Theorem \ref{4.1} has the following straightforward consequences. Let $n \ge 2$.  It is known that $\mathsf D (C_n^r)=\mathsf D^*(C_n^r)$ for $r \in [1,2]$ and if $\mathsf D (C_n^r) = \mathsf D^* (C_n^r)$ holds for some $r \ge 3$, then
$\mathsf D (C_n^s) = \mathsf D^* (C_n^s)$ for all $s \in [1,r]$. The standing conjecture is that $\mathsf D (C_n^r) = \mathsf D^* (C_n^r)$ for all $r \in \N$.

\medskip
\begin{corollary} \label{4.2}
Let $G = C_n^r$, where $n \ge 2$  and $r\geq 1$, and suppose that $\mathsf D (G) = \mathsf D^* (G)$.
Then, for every  $k\geq 1$, we have
\[
\rho_{2k+1} (G) = \begin{cases}
                  k \mathsf D (G) + 1 & \ r = 1 \\
                  k \mathsf D (G) + \Blfloor \frac{\mathsf  D (G)}{2} \Brfloor & \ r \ge 2 \,.
                  \end{cases}
\]
\end{corollary}

\begin{proof}
For $r=1$, this follows from Proposition \ref{3.2}.2.  For $r \ge 2$, it follows from Theorem \ref{4.1}.
\end{proof}

\medskip
\begin{corollary} \label{4.3}
Let  $G = C_{q_1}^{s_1} \oplus \ldots \oplus C_{q_r}^{s_r}$ be a $p$-group where $q_1, \ldots, q_r$ are powers of a fixed prime and  $s_1, \ldots, s_r \in \N_{\ge 2}$.  Then
\[
\rho_{2k+1} (G) = k \mathsf D (G) + \Blfloor \frac{\mathsf  D (G)}{2} \Brfloor \quad \mbox{ for every $k\geq 1$}\,.
\]
\end{corollary}

\begin{proof}
Since $G$ is a $p$-group, we have $\mathsf D(G) = \mathsf D^* (G)$ by Proposition \ref{2.3}, and hence the assertion follows from Theorem \ref{4.1}.
\end{proof}

\medskip
We start with the preparations for the proof of Theorem \ref{4.1}.
Let $G$ be a finite abelian group. The inequality $\rho_{3}(G)\geq \omega$ means there are  $U_1,U_2,U_3, W_1,\ldots,W_\rho\in\mathcal A(G)$ with \be\label{rho3def}U_1U_2U_3=W_1\cdot\ldots\cdot W_\rho\quad\und\quad \rho\geq \omega.\ee
For each $W_i$, where $i\in [1,\rho]$, we may write $W_i=T_{i,1}T_{i,2}T_{i,3}$ with the $T_{i,j}\mid U_j$ subsequences such that $\prod_{i=1}^{\rho}T_{i,j}=U_j$ for each $j\in [1,3]$.
There may be multiple ways to do so. If there is a way to do so with $|T_{i,j}|\leq 1$ for all $i$ and $j$, then we say that the factorization \eqref{rho3def} is \emph{weakly reduced}.
Let
$W'_i$ be the sequence obtained from $W_i=T_{i,1}T_{i,2}T_{i,3}$ by replacing each nonempty $T_{i,j}$ with the sum of its terms. Likewise, let $U'_j$ be the sequence obtained from $U_j=\prod_{i=1}^{\rho}T_{i,j}$ by replacing each nonempty $T_{i,j}$ with the sum of its terms.
The sequence $X=\prod_{i=1}^\rho |W'_i|\in \Fc(\Z)$ is called a \emph{spread} for the factorization \eqref{rho3def}, and it depends on the $T_{i,j}$ (so a given factorization \eqref{rho3def} may have multiple spreads). It is readily seen that if $U\in \mathcal A(G)$ is an atom and $T\mid U$ is nontrivial, then the sequence $UT^{-1}\sigma(T)$, obtained by replacing the terms from $T$ with their sum, is also an
 atom. Hence the $W'_i$ and $U'_i$ are atoms with $U'_1U'_2U'_3=W'_1\cdot\ldots\cdot W'_\rho$ a weakly reduced factorization having $|W'_i|\leq 3$ for all $i$. Thus, when $\rho_3(G)\geq \omega$, we may always assume our factorization  \eqref{rho3def} is weakly reduced.
 We say that \emph{$\rho_3(G)\geq \omega$ with spread $X\in \Fc(\{1,2,3\})$} if there exists a factorization \eqref{rho3def} having spread $X$, in which case, per the argument above, we may also assume there is a weakly reduced factorization having spread $X$.

 If $1\in \supp (X)$, then  some $W_i$,  say $W_1$, has $W_1=T_{1,j}$ for some $j$, implying that the zero-sum  sequence $W_1=T_{1,j}$ is a subsequence of $U_j$. As $U_j$ is an atom, this is only possible if $W_1=T_{1,j}=U_j$, which forces all other $T_{i,j}$ with $i\neq 1$ to be empty in view of $\prod_{i=1}^\rho T_{i,j}=U_j$. In particular, $|W'_i|\leq 2$ for all $i$ when $1\in \supp(X)$, meaning we cannot have both $1,\,3\in \supp(X)$. From this, we see that we have three mutually exclusive possibilities for a spread $X$: $$1\in \supp(X)\quad\mbox{ or } \quad  \supp(X)=\{2\}\quad\mbox{ or } \quad 3\in \supp(X).$$ Note there is a spread $X$ with $1\in \supp(X)$ precisely when $W_s=U_t$ for some $s\in [1,\rho]$ and $t\in [1,3]$. If $0\in \supp(U_1U_2U_3)$, then $1$ will be a term in any spread $X$. If $\supp(X)=\{2\}$ and \eqref{rho3def} is weakly reduced, so that $|W_i|=2$ for all $i$, then $U_1$ must have a subsequence $x y\mid U_1$ with $-x\in \supp(U_2)$ and $-y\in \supp(U_3)$, with similar statements holding for $U_2$ and $U_3$. When $3\in \supp(X)$, we refer to a $W_i$ with $|W'_i|=3$ as a \emph{traversal} for the factorization \eqref{rho3def}.

\begin{lemma} \label{4.4}
Let $G = G_1 \oplus G_2$ be a finite  abelian group where   $G_1, G_2 \subset G$ are nontrivial subgroups with
$\mathsf D(G_1)\leq \mathsf D(G_2)$. Then $\rho_3(G)\geq 2\mathsf D(G_1)+\mathsf D(G_2)-2$  with spread $X\in \Fc(\{2,3\})$ having  $\vp_3(X)=1$.
\end{lemma}

\begin{proof}
Let $$V=v_0\cdot\ldots\cdot v_{\mathsf d(G_1)}\in \mathcal A(G_1)\quad\und\quad W=w_0\cdot\ldots\cdot w_{\mathsf d(G_2)}\in \mathcal A(G_2)$$ be atoms with maximal lengths $|V|=\mathsf d(G_1)+1=\mathsf D(G_1)$ and $|W|=\mathsf d(G_2)+1=\mathsf D(G_2)$. Since $G_1$ and $G_2$ are nontrivial, $0\notin \supp(VW)$. Let $V'=v_1\cdot\ldots\cdot v_{\mathsf d(G_1)}$ and
$W'=w_1\cdot\ldots\cdot w_{\mathsf d(G_1)}$ and define
\[
\begin{aligned}
U_1 & =   (v_0+w_0)V'W'\prod_{i=\mathsf d(G_1)+1}^{\mathsf d(G_2)}w_i=(Vv_0^{-1})(Ww_0^{-1})(v_0+w_0), \\
U_2 & =   (-w_0)(-V')\prod_{i=1}^{\mathsf d (G_1)} (v_i-w_i)\prod_{i=\mathsf d(G_1)+1}^{\mathsf d(G_2)}(-w_i), \quad \text{and}  \\
U_3 & =   (-v_0)(-W')\prod_{i=1}^{\mathsf d (G_1)} (w_i-v_i)  .
\end{aligned}
\]
It is easily seen that  $U_1, U_2, U_3 \in \mathcal A (G)$ are atoms with $(U_1(v_0+w_0)^{-1})(U_2(-w_0)^{-1})(U_3(-v_0)^{-1})$ having a factorization into atoms of length $2$, evidencing that $\rho_3(G)\geq 1+\frac{4\mathsf d(G_1)+2\mathsf d(G_2)}{2}=2\mathsf D(G_1)+\mathsf D(G_2)-2$ with
$(v_0+w_0)(-w_0)(-v_0)$ the unique traversal.
\end{proof}

\medskip
\begin{lemma} \label{4.5}
Let $G = G_1 \oplus G_2$ be a finite abelian group where $G_1, G_2 \subset G$ are nontrivial subgroups such that $\rho_3(G_1)\geq \omega_1$ and $\rho_3(G_2)\geq \omega_2$ both hold with respective spreads $X,\,Y\in \Fc(\{2,3\})$.
\begin{itemize}
 \item[1.] If $\vp_3(X)+\vp_3(Y)\geq 1$, then $\rho_3(G)\geq \omega_1+\omega_2-1$ holds with spread $Z\in \Fc(\{2,3\})$ having $\vp_3(Z)=\vp_3(X)+\vp_3(Y)-1$.
     \item[2.] If $\supp(X)=\supp(Y)=\{2\}$, then  $\rho_3(G)\geq \omega_1+\omega_2-2$ holds with spread $Z\in \Fc(\{2,3\})$ having $\vp_3(Z)=1$.
\end{itemize}
\end{lemma}

\begin{proof}
For $i \in \{1,2\}$, let $\pi_i \colon G = G_1 \oplus G_2 \to G_i$ denote the canonical projection.

1.  First suppose $\vp_3(X)=r\geq 1$ and $\vp_3(Y)=s\geq 1$. Let $V_1,\,V_2,\,V_3\in \mathcal A(G_1)$ and $W_1,\,W_2,\,W_3\in \mathcal A(G_2)$ be atoms having weakly reduced factorizations
$$V_1V_2V_3=X_1\cdot\ldots\cdot X_{\rho_1}\quad\und\quad W_1W_2W_3=Y_1\cdot\ldots\cdot Y_{\rho_2}$$ with the $X_i\in \mathcal A(G_1)$, the $Y_i\in\mathcal A(G_2)$, \ $\rho_i\geq \omega_i$ for $i\in [1,2]$, and $X_i$ and $Y_j$ traversals in their respective factorizations for $i\in [1,r]$ and $j\in [1,s]$.  In particular, $|X_i|=|Y_j|=2$ for $i\geq r+1$ and $j\geq s+1$, \ $X_1=a_1a_2a_3$ with $a_i\in \supp(V_i)$ for $i\in [1,3]$, and  $Y_1=b_1b_2b_3$ with $b_i\in \supp(W_i)$ for $i\in [1,3]$.
 Let $V'_i=V_ia_i^{-1}$ and $W'_i=W_ib_i^{-1}$ for $i\in [1,3]$.
Now we define
\[
\begin{aligned}
U_1 & =  V'_1W'_1(a_1+b_1)=(V_1a_1^{-1})(W_1b_1^{-1})(a_1+b_1), \\
U_2 & =  V'_2W'_2(a_2+b_2)=(V_2a_2^{-1})(W_2b_2^{-1})(a_2+b_2) \quad \text{and} \\
U_3 & =  V'_3W'_3(a_3+b_3)=(V_3a_3^{-1})(W_3b_3^{-1})(a_3+b_3).
\end{aligned}
\] It is easily observed that the $U_i$ are atoms. Moreover, $V'_1V'_2V'_3=(V_1V_2V_3)X_1^{-1}=X_2\cdot\ldots\cdot X_{\rho_1}$ and
  $W'_1W'_2W'_3=(W_1W_2W_3)Y_1^{-1}=Y_2\cdot\ldots\cdot Y_{\rho_2}$. Thus  $U_1U_2U_3=X_2\cdot\ldots X_{\rho_1}Y_2\cdot\ldots\cdot Y_{\rho_2}W$ with $W=(a_1+b_1)(a_2+b_2)(a_3+b_3)$ a traversal in view of $a_1+a_2+a_3+b_1+b_2+b_3=\sigma(X_1)+\sigma(Y_1)=0$. Moreover, $X_2,\ldots,X_s,Y_2,\ldots,Y_r$  also remain traversals in this factorization, while no $X_i$ nor $Y_j$ with $i\geq r+1$ or $j\geq s+1$ can be a traversal in view of $|X_i|=|Y_j|=2$. Thus  $\rho_3(G)\geq \rho_1+\rho_2-1\geq \omega_1+\omega_2-1$ holds with spread $Z$ having $\vp_3(Z)=\vp_3(X)+\vp_3(Y)-1>0$, ensuring $Z\in \Fc(\{2,3\})$ (noted before Lemma \ref{4.4}).

\bigskip

Next suppose that either $\vp_3(X)>0=\vp_3(Y)$ or $\vp_3(Y)>0=\vp_3(X)$, say w.l.o.g. the former, so $$\vp_3(X)=r>0\quad\und\quad \supp(Y)=\{2\},$$  the latter in view of $Y\in \Fc(\{2,3\})$.
  Let $V_1,\,V_2,\,V_3\in \mathcal A(G_1)$ and $W_1,\,W_2,\,W_3\in \mathcal A(G_2)$ be atoms having weakly reduced factorizations
$$V_1V_2V_3=X_1\cdot\ldots\cdot X_{\rho_1}\quad\und\quad W_1W_2W_3=Y_1\cdot\ldots\cdot Y_{\rho_2}$$ with the $X_i\in \mathcal A(G_1)$, the $Y_i\in\mathcal A(G_2)$, \ $\rho_i\geq \omega_i$ for $i\in [1,2]$, the  $X_i$ with $i\in [1,r]$  traversals in their factorization, and \ $|Y_i|=2$ and $|X_j|=2$ for all $i\in [1,\rho_2]$ and $j\geq r+1$. In particular,
 $X_1=a_1a_2 a_3$ with $a_i\in \supp(V_i)$ for $i\in [1,3]$.
Since  $1\notin \supp(XY)$,  we have $0\notin\supp(V_1V_2V_3W_1W_2W_3)$ implying $|V_i|,\,|W_i|\geq 2$  for all $i\in [1,3]$.
Also, as discussed before Lemma \ref{4.4}, there must be
 a length two subsequence $xy\mid W_1$ with $-x\mid W_2$ and $-y\mid W_3$.
 Now we define
\[
\begin{aligned}
U_1 & =  (V_1a_1^{-1}) (W_1x^{-1}y^{-1}) (x-a_2)(y+a_1+a_2) \,, \\
U_2 & =  (V_2a_2^{-1})(W_2(-x)^{-1})(a_2-x)\quad \text{and} \\
U_3 & =  (V_3a_3^{-1})(W_3(-y)^{-1})(a_3-y) \,.
\end{aligned}
\]
Obviously, we have $U_2, U_3 \in \mathcal A (G)$ and $U_1\in \mathcal B(G)$. Letting $S = U_1(y+a_1+a_2)^{-1}$ and considering $\pi_2 (S)$ and $\pi_1 (S)$ shows that  $S$ is zero-sum free, implying that  $U_1 \in \mathcal A (G)$. Since $a_1+a_2+a_3=\sigma(X_1)=0$, we have $(y+a_1+a_2)+(a_3-y)=0$. Thus, since $(V_1a_1^{-1})(V_2a_2^{-1})(V_3a_3^{-1})=(V_1V_2V_3)X_1^{-1}=X_2\cdot\ldots\cdot X_{\rho_1}$ and since $W_1W_2W_3$ has a  factorization into $\rho_2$  atoms of length $2$, it is now clear that $U_1U_2U_3$ has a factorization using $(\rho_1-1)+(\rho_2-2)+2\geq \omega_1+\omega_2-1$ atoms, say
\be\label{daisy}U_1U_2U_3=X_2\cdot\ldots\cdot X_{\rho_1}Z_1\cdot\ldots\cdot Z_{\rho_2}\ee  with $|Z_i|=2$ for all $i$. Hence $\rho_3(G)\geq \rho_1+\rho_2-1\geq \omega_1+\omega_2-1$.
Moreover, each $X_i$ with $i\in [2,r]$ remains a traversal for \eqref{daisy}, while this cannot be the case for $X_j$ with $j\geq r+1$ nor any $Z_i$ as $|X_j|=|Z_i|=2$ for $j\geq s+1$ and all $i$. Thus \eqref{daisy} has a spread $Z$ with $$\vp_3(Z)=\vp_3(X)-1=\vp_3(X)+\vp_3(Y)-1.$$ If $\vp_3(X)\geq 2$, this shows $3\in \supp(Z)$, whence $Z\in \Fc(\{2,3\})$ as discussed before Lemma \ref{4.4}. On the other hand, if $\vp_3(X)=1$, then all atoms in the factorization \eqref{daisy} have length $2$.  Thus, since $|U_j|=|V_j|+|W_j|-1\geq 3$ for all $j\in [1,3]$, we see that none of these atoms of length two can equal some $U_j$, meaning $1\notin \supp(Z)$ for any spread $Z$ for \eqref{daisy}, also explained above Lemma \ref{4.4}. In this case, $\supp(Z)=\{2\}$, completing the proof of Part 1.

\bigskip

2. Suppose $\supp(X)=\supp(Y)=\{2\}$.
Let $V_1,\,V_2,\,V_3\in \mathcal A(G_1)$  be atoms such that $V_1V_2V_3$ has a weakly reduced factorization into $\rho_1\geq \omega_1$ atoms of length $2$ and let  $W_1,\,W_2,\,W_3\in \mathcal A(G_2)$  be atoms  such that $W_1W_2W_3$ has a weakly reduced factorization into $\rho_2\geq \omega_2$ atoms of length $2$. As explained before Lemma \ref{4.4}, we may assume there is a length $2$ subsequence  $xy\mid W_1$ with $-x\mid W_2$ and $-y\mid W_3$ and a length $2$ subsequence $ab\mid V_2$ with $-a\mid V_1$ and $-b\mid V_3$. Now we define
\[
\begin{aligned}
U_1 & =  (V_1(-a)^{-1}) (W_1x^{-1}y^{-1}) (x-a)(y) \,, \\
U_2 & =  (V_2a^{-1}b^{-1})(W_2(-x)^{-1})(a-x)(b)\quad \text{and} \\
U_3 & =  (V_3(-b)^{-1})(W_3(-y)^{-1})(-b-y) \,.
\end{aligned}
\]
Obviously, we have $U_3 \in \mathcal A (G)$ and $U_1,\,U_2\in \mathcal B(G)$. Letting $S = U_1y^{-1}$ and considering $\pi_2 (S)$ and $\pi_1 (S)$ shows that  $S$ is zero-sum free, implying that  $U_1 \in \mathcal A (G)$. Likewise, letting $T=U_2b^{-1}$ and considering $\pi_1 (T)$ and $\pi_2 (T)$ shows that  $T$ is zero-sum free, implying that  $U_2 \in \mathcal A (G)$. Let $c_1=y$, $c_2=b$ and $c_3=-b-y$. Since $V_1V_2V_3$ and $W_1W_2W_3$ both have factorizations into  atoms of length $2$, it is now clear that $(U_1c_1^{-1})(U_2c_2^{-1})(U_3c_3^{-1})$ has a factorization into $(\rho_1-2)+(\rho_2-2)+1=\rho_1+\rho_2-3$ atoms of length $2$, which together with the unique traversal $c_1c_2c_3$ gives a factorization of $U_1U_2U_3$ into $\rho_1+\rho_2-2$ atoms,  showing that $\rho_3(G)\geq \rho_1+\rho_2-2\geq \omega_1+\omega_2-2$ holds with spread $Z$ having $\vp_3(Z)=1$, ensuring $Z\in \Fc(\{2,3\})$ (noted before Lemma \ref{4.4}).
\end{proof}

\medskip
\begin{lemma} \label{4.6}
Let $G = C_n^3$ with $n \ge 2$. Then $\rho_3(G)\geq \mathsf D^*(G)+\lfloor\frac{\mathsf D^*(G)}{2}\rfloor$ with spread $X\in \Fc(\{2,3\})$. Moreover,  $\vp_3(X)=1$ if $\mathsf D^*(G)$ is odd, and $\supp(X)=\{2\}$ if $\mathsf D^*(G)$ is even.
\end{lemma}

\begin{proof}
Let $\{e_1, e_2, e_3\}$ be a basis of $G$ and for $i \in [1,3]$, let $\pi_i \colon G = \langle e_1 \rangle \oplus \langle e_2 \rangle
\oplus \langle e_3 \rangle \to \langle e_i \rangle$ denote the canonical projection. Note that $\mathsf D^*(G)=3n-2\equiv n\mod 2$.  We handle two cases.

\smallskip
\noindent
CASE 1: \ $n$ is odd.

Then $\mathsf D^* (G) = 3n-2 \ge 7$ is odd.  We define
\[
\begin{aligned}
U_1 & = e_1^{n-1} e_2^{n-1} (e_1+e_2+e_3)^{n-1} c_1  \quad \text{with} \quad c_1 = 2e_1+2e_2+e_3  \,, \\
U_2 & = (-e_1)^{n-1}e_3^{n-1}(-e_1-e_2-e_3)^{\frac{n-1}{2}} (e_1-e_2)^{\frac{n-1}{2}} c_2  \quad \text{with} \quad c_2 = -e_1-e_2+\frac{n+1}{2}e_3 \quad \text{and}  \\
U_3 & = (-e_3)^{n-1}(-e_2)^{n-1}(-e_1-e_2-e_3)^{\frac{n-1}{2}} (-e_1+e_2)^{\frac{n-1}{2}} c_3  \quad \text{with} \quad c_3 = -e_1-e_2+\frac{n-3}{2}e_3 \,.
\end{aligned}
\]
Clearly, $U_i\in \mathcal B(G)$ for $i\in [1,3]$. Considering $\pi_3 (U_1c_1^{-1})$, $\pi_2 (U_2c_2^{-1})$ and $\pi_1 (U_3c_3^{-1})$, we infer that  the sequences $U_ic_i^{-1}$ are zero-sum free for every $i \in [1,3]$. Therefore, we have $U_1, U_2, U_3 \in \mathcal A (G)$, and it is now easily seen that $(U_1c_1^{-1})(U_2c_2^{-1})(U_3c_3^{-1})$ has a factorization into atoms of length $2$, which together  with the unique traversal $c_1c_2c_3$ shows that  $\rho_3(G)\geq 1+\frac{9n-9}{2}=\mathsf D^*(G)+\lfloor\frac{\mathsf D^*(G)}{2}\rfloor$ holds with spread $X$ having $\vp_3(X)=1$, so that $X\in \Fc(\{2,3\})$ as noted before Lemma \ref{4.4}.

\smallskip
\noindent
CASE 2: \ $n$ is even.

Then $\mathsf D^* (G) = 3n-2 \ge 4$ is even. We define
\[
\begin{aligned}
U_1 & = e_1^{n-1} e_2^{n-1} (e_1+e_2+e_3)^{n-2}(2e_1+e_2+e_3)(e_1+2e_2+e_3) \,, \\
U_2 & = (-e_1)^{n-1} e_3^{n-1}(-e_1-e_2-e_3)^{n/2-1}(e_1-e_2+e_3)^{n/2-1}(-2e_1-e_2-e_3)(e_1-e_2+2e_3) \quad \text{and} \\
U_3 & = (-e_3)^{n-1}(-e_2)^{n-1}(-e_1-e_2-e_3)^{n/2-1}(-e_1+e_2-e_3)^{n/2-1}(-e_1-2e_2-e_3)(-e_1+e_2-2e_3)  \,.
\end{aligned}
\]
Clearly, $U_i\in \mathcal B(G)$ for $i\in [1,3]$.
Considering $\pi_3 (U_1)$, $\pi_2 (U_2)$ and $\pi_1 (U_3)$, we infer that $U_1, U_2, U_3 \in \mathcal A (G)$. By construction, $U_1U_2U_3$
has a factorization into atoms of length $2$, say $U_1U_2U_3=Z_1\cdot\ldots\cdot Z_{\frac12 |U_1U_2U_3|}$, implying that
$\rho_3(G) \geq
\frac{1}{2} |U_1U_2U_3| = \frac{3(3n-2)}{2}=\mathsf D^*(G)+\lfloor\frac{ \mathsf D^*(G)}{2}\rfloor$. Moreover, since $|U_i|=3n-2>2=|Z_j|$ for all $i$ and $j$, we see that $1\notin \supp(X)$ in any spread $X$, whence $\supp(X)=\{2\}$, completing the proof.
\end{proof}

\medskip
\begin{proof}[Proof of Theorem \ref{4.1}]
By Proposition \ref{2.1}, we have $\rho_k (H) = \rho_k (G)$ for all $k\geq 1$.
By Lemma \ref{3.1}.1 and Lemma \ref{3.1}.2, it suffices to prove the assertion for $k=1$.
By hypothesis,  $G$ can be written in the form
\[
G =  C_{m_1}^{t_1} \oplus \ldots \oplus C_{m_{\alpha}}^{t_{\alpha}} \,,
\]
where $\{m_1, \ldots, m_{\alpha} \} = \{n_1, \ldots, n_r\}$ with $t_i \in \{2,3\}$.
We proceed by induction on $\alpha$ to show that $\rho_3(G)\geq \mathsf D^*(G)+\lfloor\frac{\mathsf D^*(G)}{2}\rfloor$ holds with spread $X\in \Fc(\{2,3\})$ with $\vp_3(X)=1$ when $\mathsf D^*(G)$ is odd and with $\supp(X)=\{2\}$ when $\mathsf D^*(G)$ is even, which will complete the proof.

Since $n_1\mid\ldots \mid n_r$ with $\{m_1, \ldots, m_{\alpha} \} = \{n_1, \ldots, n_r\}$, we have  \be\nn\mathsf D^*(G)=\mathsf d^*(C_{m_1}^{t_1})+\mathsf d^*(C_{m_2}^{t_2} \oplus \ldots \oplus C_{m_{\alpha}}^{t_{\alpha}})+1=\mathsf D^*(K)+\mathsf D^*(L)-1,\ee where $K=C_{m_1}^{t_1}$ and $L= C_{m_2}^{t_2} \oplus \ldots \oplus C_{m_{\alpha}}^{t_{\alpha}}$. If $t_1=2$, then $\mathsf D^*(K)=\mathsf D^*(C_{m_1}^2)=2m_1-1$ is odd and Lemma \ref{4.4} implies that $\rho_3(K)=\rho_3(C_{m_1}^{2})\geq 3m_1-2=(2m_1-1)+\lfloor\frac{2m_1-1}{2}\rfloor=\mathsf D^*(K)+\lfloor\frac{\mathsf D^*(K)}{2}\rfloor$ with spread $X$ having $\vp_3(X)=1$, so that $X\in \Fc(\{2,3\})$.
  If $t_1=3$, then  Lemma \ref{4.6} implies that $\rho_3(K)=\rho_3(C_{m_1}^{3})\geq \mathsf D^*(K)+\lfloor\frac{\mathsf D^*(K)}{2}\rfloor$ with spread $X\in\Fc( \{2,3\})$. Moreover, $\vp_3(X)=1$ if $\mathsf D^*(K)$ is odd, and $\supp(X)=\{2\}$ if $\mathsf D^*(K)$ is even. This completes the base case when $\alpha=1$. Thus we may assume $\alpha\geq 2$, in which case the induction hypothesis ensures that $$\rho_3(K)\geq \mathsf D^*(K)+\lfloor\frac{\mathsf D^*(K)}{2}\rfloor\quad\und\quad \rho_3(L)\geq \mathsf D^*(L)+\lfloor\frac{\mathsf D^*(L)}{2}\rfloor$$ with respective spreads $X,\,Y\in \Fc(\{2,3\})$.

  If $\mathsf D^*(K)$ and $\mathsf D^*(L)$ are both even, then $\mathsf D^*(G)=\mathsf D^*(K)+\mathsf D^*(L)-1$ is odd and $\supp(X)=\supp(Y)=\{2\}$, whence  Lemma \ref{4.5}.2  yields $$\rho_3(G)\geq \mathsf D^*(K)+\frac{\mathsf D^*(K)}{2}+ \mathsf D^*(L)+\frac{\mathsf D^*(L)}{2}-2=\mathsf D^*(G)+\frac{\mathsf D^*(G)-1}{2}$$ with spread $Z\in \Fc(\{2,3\})$ having $\vp_3(Z)=1$, as desired. If $\mathsf D^*(K)$ and $\mathsf D^*(L)$ are both odd, then $\vp_3(X)=\vp_3(Y)=1$ and $\mathsf D^*(G)=\mathsf D^*(K)+\mathsf D^*(L)-1$ is odd, whence  Lemma \ref{4.5}.1  yields $$\rho_3(G)\geq \mathsf D^*(K)+\frac{\mathsf D^*(K)-1}{2}+ \mathsf D^*(L)+\frac{\mathsf D^*(L)-1}{2}-1=\mathsf D^*(G)+\frac{\mathsf D^*(G)-1}{2}$$ with spread $Z\in \Fc(\{2,3\})$ having $\vp_3(Z)=\vp_3(X)+\vp_3(Y)-1=1$, as desired. Finally, if $\mathsf D^*(K)$ and $\mathsf D^*(L)$ have different parities, then $\vp_3(X)+\vp_3(Y)=1$, \ $\mathsf D^*(G)=\mathsf D^*(K)+\mathsf D^*(L)-1$ is even, and    Lemma \ref{4.5}.1  yields $$\rho_3(G)\geq \mathsf D^*(K)+\frac{\mathsf D^*(K)}{2}+ \mathsf D^*(L)+\frac{\mathsf D^*(L)}{2}-\frac12-1=\mathsf D^*(G)+\frac{\mathsf D^*(G)}{2}$$ with spread $Z\in \Fc(\{2,3\})$ having $\vp_3(Z)=\vp_3(X)+\vp_3(Y)-1=0$, forcing $\supp(Z)=\{2\}$. This completes the induction. When $\mathsf D^*(G)=\mathsf D(G)$, the needed upper bound comes from Lemma \ref{3.1}.2
\end{proof}

\bigskip
\section{Groups of rank two} \label{5}
\bigskip

The aim of this section is to prove the following characterization. It provides the first non-cyclic groups $G$ at all for which  $\rho_{2k+1}(G)$ is strictly smaller than the upper bound $k \mathsf D (G) + \lfloor \mathsf D (G)/2 \rfloor$ for some $k \in \N$.

\medskip
\begin{theorem} \label{5.1}
Let $H$ be a Krull monoid with finite class group $G$ such that every class contains a prime divisor. Suppose that
$G = C_{m} \oplus C_{mn}$ with $n \geq 1$ and $m \ge 2$. Then
\[
\rho_3 (H) = \mathsf D (G) + \Blfloor \frac{\mathsf D (G)}{2} \Brfloor \qquad \text{if and only if} \qquad
n = 1 \ \ \text{or} \ \ m=n=2 \,.
\]
\end{theorem}

\medskip
We start with two corollaries providing examples of groups $G$ having rank two which show that Theorem \ref{5.1} is sharp in two aspects. Indeed, Corollary \ref{5.2} shows that  these groups $G$ satisfy
\[
\rho_3 (G) = \mathsf D (G) + \Blfloor \frac{\mathsf D (G)}{2} \Brfloor - 1 \quad \text{but} \quad \rho_{2k+1} (G) = k \mathsf D (G) + \Blfloor \frac{\mathsf D (G)}{2} \Brfloor \quad \text{for all} \ k \ge 2 \,.
\]
After that, we deal with groups of the form $G = C_2 \oplus C_{2n}$ where $n \ge 3$.
Since for cyclic groups $G$ we have $\rho_{2k+1}(G)=k\mathsf D(G)+1$ for all $k \ge 1$, groups of the form $C_2 \oplus C_{2n}$ are the canonical first choice for testing Conjecture {\bf C2}. Indeed, we verify  Conjecture {\bf C2} for them and show that there exists an integer $k^* \in \N$ (by Theorem \ref{5.1} we must have $k^* > 1$ for $n>2$)
such that
\[
\rho_{2k+1} (G) = k \mathsf D (G) + \Blfloor \frac{\mathsf D (G)}{2} \Brfloor \quad \text{for all} \ k \ge k^* \,.
\]
Moreover, Corollary \ref{5.3} provides the first example of a group where, for some odd $k \in \N$, strict inequalities hold in the crucial inequality (\ref{crucialinequality}).

\medskip
\begin{corollary} \label{5.2}
Let $G = C_m \oplus C_{2m}$ with $m \ge 2$.
\begin{enumerate}
\item  If $m = 2$, then $\rho_{2k+1} (G) = k\mathsf D(G)+\lfloor \frac{ \mathsf D (G)}{2} \rfloor$ for every $k \ge 1$.

\smallskip
\item  If $m \ge 3$, then $\rho_5 (G) \ge 2 \mathsf D (G) + (m+1)$.

\smallskip
\item If $m \in \{3,4\}$, then
\[
\rho_3 (G) = \mathsf D (G) + \Blfloor \frac{\mathsf D (G)}{2} \Brfloor - 1 \quad \text{and} \quad \rho_{2k+1} (G) = k \mathsf D (G) + \Blfloor \frac{\mathsf D (G)}{2} \Brfloor \quad \text{for all} \ k \ge 2 \,.
\]
\end{enumerate}
\end{corollary}

\begin{proof}
Let $\{e_1, e_2\}$ be a basis of $G$ with $\ord (e_1) = m$ and $\ord (e_2) = 2m$. Then $\mathsf D (G) = 3m-1$.

\smallskip
1. We define
\[
U_1 = e_1e_2(e_1+e_2)^3 \,,  \quad U_2 = e_1(-e_2)^3(e_1-e_2) \quad \text{and} \quad U_3 = e_2^2(e_1-e_2)^2 \,.
\]
Obviously, $U_1U_2U_3$ may be written as a product of $7$ which implies that $\rho_3 (G) = \mathsf D(G)+\lfloor \frac{ \mathsf D (G)}{2} \rfloor$. Now the assertion  follows from  Lemma \ref{3.1}.3.

\smallskip
2. We define
\begin{align*} & U_1 = e_1^{m-1}e_2^{2m-1}(e_1+e_2),  && U_2 = (-e_1)^{m-1}e_2^{2m-1}(-e_1+e_2) \\
& U_3  = (e_1+e_2)^{m-1}(-e_2)^{2m-1}(e_1+me_2) \quad \und \quad  && U_4 = (-e_1-e_2)^{2m-1}(-e_1+me_2)^2 (e_1-e_2).\end{align*}
Then $U_1, U_2, U_3, U_4 \in \mathcal A (G)$ (note we need $m\geq 3$ to ensure $U_4\in \mathcal A(G)$), $|U_1| = |U_2| = |U_3| = \mathsf D (G)$ and $|U_4| = 2m+2$). By construction, $U_1 U_2 U_3^2 U_4$
has a factorization into atoms of length $2$, which implies that
\[
\rho_5(G)\geq  \frac{|U_1 U_2 U_3^2 U_4|}{2} = 2 \mathsf D (G) + (m+1) \,.
\]

\smallskip
3. Proposition \ref{3.5}.1 and Theorem \ref{5.1} imply that
\[
\mathsf D (G) + m \le \rho_3 (G) \le \mathsf D (G) + \Blfloor \frac{\mathsf D (G)}{2} \Brfloor - 1 \,,
\]
which is an equality because $m \in \{3,4\}$. By Lemma \ref{3.1}.3, it suffices to show that
\[
\rho_5 (G) \ge 2 \mathsf D (G) + \Blfloor \frac{\mathsf D (G)}{2} \Brfloor \,,
\]
which follows from  2. above because $m \in \{3,4\}$ ensures $\Blfloor \frac{\mathsf D (G)}{2} \Brfloor=m+1$.
\end{proof}

\medskip
\begin{corollary} \label{5.3}
Let $G = C_2 \oplus C_{2n}$ with $n \ge 3$. Then
\[
\mathsf D(G)+1 < \rho_3 (G) < \mathsf D(G)+ \Blfloor \frac{\mathsf  D (G)}{2} \Brfloor \quad \text{and} \quad
\rho_{2k+1} (G) = k \mathsf D (G) + \Blfloor \frac{\mathsf  D (G)}{2} \Brfloor \quad \text{for every} \  k\geq 2n-1 \,.
\]
\end{corollary}

\begin{proof}
We have $\mathsf D(G)=\mathsf D^* (G)=2n+1$. 
The left inequality follows from Proposition \ref{3.5}.2 and from Proposition \ref{3.2}.1, and the right inequality follows from Theorem \ref{5.1}.

To prove the second statement, let $\{e_1, e_2\}$ be a basis of $G$ with $\ord (e_1) = 2$ and $\ord (e_2) = 2n$. For $i \in [1, n]$, we define
\[
U_i = e_2^{2n-1} \bigl( e_1 - (i-1)e_2 \bigr) \bigl( e_1 + i e_2 \bigr) \in \mathcal A (G) \,.
\]
Let
\[
V_1 = (e_1+e_2)^{2n-1}e_2e_1 \in \mathcal A (G) \,.
\] Let $W=e_2(e_1+e_2)(e_1-2e_2)$.
By construction, $S=\Bigl( U_2^2 (-U_1)^2 \Bigr) \cdot \ldots \cdot \Bigl( U_n^2 (-U_1)^2 \Bigr)\Bigl( U_1 (-U_2)V_1 \Bigr)$ is a product of $4(n-1) + 3=4n-1$ atoms and $SW^{-1}$
has a factorization into atoms of length $2$. This implies that
\[
\rho_{2(2n-1)+1}(G)\geq  1+\frac{|S|-3}{2} =1+\frac{(4n-1)(2n+1)-3}{2}= (2n-1)\mathsf D(G)+ \Blfloor \frac{\mathsf D(G)}{2}\Brfloor.
\] The result now follows from  Lemma \ref{3.1}.3.
\end{proof}

\smallskip
The proof of Theorem \ref{5.1} is based on the recent characterization of minimal zero-sum sequences of maximal length in groups of rank two, which will be formulated in Main Proposition \ref{5.4}. The proof of the characterization is obtained by combining the main results from  \cite{Ga-Ge03b}, \cite{Ga-Ge-Gr10a}, \cite{Re10c}, \cite{Sc10b} with a few small order groups handled by direct computation \cite{Bh-Ha-SP10b}.  The version below  is derived from this original in a few short lines
\cite[Theorem 3.1]{B-G-G-P13a} (apart from (e) and the fact that both parts of (d) hold when $n=2$, which we will deduce from the rest of theorem in the explanations below). It eliminates some overlap between type I and II in the original statement.

\medskip
\begin{mproposition} \label{5.4}
Let  $G = C_{m} \oplus C_{mn}$  with $n\geq 1$  and $m \ge 2$.  A sequence $S$
over $G$ of length $\mathsf D (G) = m+mn-1$ is a minimal zero-sum
sequence if and only if it has one of the following two forms{\rm
\,:}
\begin{itemize}
\medskip
\item \[
      S = e_1^{\ord (e_1)-1} \prod_{i=1}^{\ord (e_2)}
      (x_{i}e_1+e_2),
      \]where
      \begin{itemize}\item[(a)] $\{e_1, e_2\}$ is a basis of $G$,
      \item[(b)] $x_1, \ldots, x_{\ord (e_2)}  \in
      [0, \ord (e_1)-1]$ and $x_1 + \ldots + x_{\ord (e_2)} \equiv 1
      \mod \ord (e_1)$. \end{itemize} In this case, we say that $S$ is of type I(a) or I(b) according to whether $\ord(e_2)=m$ or $\ord(e_2)=mn>m$.

\medskip
\item \[
      S = f_1^{sm - 1} f_2^{(n-s)m+\epsilon}\prod_{i=1}^{m-\epsilon} ( -x_{i} f_1 +
      f_2),
      \] where
\begin{itemize}
      \item[(a)] $\{f_1, f_2\}$ is a generating set for  $G$ with $\ord (f_2) =
      mn$ and $\ord(f_1)>m$,
\item[(b)] $\epsilon\in [1,m-1]$  and
       $s \in [1, n-1]$,
        \item[(c)] $x_1, \ldots, x_{m-\epsilon} \in [1, m-1]$ with $x_1 + \ldots + x_{m-\epsilon} = m-1$,  \item[(d)] either  $s=1$ or
      $mf_1 = mf_2$, with both holding when $n=2$, and
      \item[(e)] either $\epsilon\geq 2$  or $mf_1\neq mf_2$.\end{itemize} In this case, we say that $S$ is of type II.
\end{itemize}
\end{mproposition}

\medskip
We gather some simple consequences of the above characterization which will be used without further mention. Let all notation be as in the Main Proposition \ref{5.4}.

\medskip

It is easy to see that $|\supp (S)|\ge3$.

\medskip

When $S$ has type II, it is always possible to find some $f'_1\in G$ such that $\{f'_1,f_2\}$ is a basis for $G$ with $\ord(f'_1)=m$ and $f_1=f'_1+\alpha f_2$ for some $\alpha\in [1,mn-1]$ (see \cite{B-G-G-P13a}). In particular,
since $mf_1\neq 0$ (in view of $\ord(f_1)>m$), we have $\ord(f_1)=tm$ for some $t\geq 2$ with $t\mid n$. Moreover, it is now readily checked that, regardless of whether $S$ has type I or II, every term of $S$ must have its order being a multiple of $m$.

\medskip

When $S$ has type II, it is clear that $-x_if_1+f_2=-x_if'_1+(1-\alpha x_i)f_2\neq f_2$ in view of $x_i\in [1,m-1]$, for any $i\in [1,m-\epsilon]$. Likewise, a term $-x_if_1+f_2=-x_if'_1+(1-\alpha x_i)f_2$ could only equal $f_1=f'_1+\alpha f_2$ if $x_i=m-1$ and $1-\alpha(m-1)=1-\alpha x_i\equiv \alpha\mod mn$, implying $1\equiv \alpha m\mod mn$, which is not possible. Consequently, we see that a term $-x_if_1+f_2$ can never equal $f_1$ or $f_2$. Likewise, since $\ord(f_1)\geq 2m$,  $-x_if_1+f_2=-x_jf_1+f_2$ is only possible if $x_i=x_j\in [1,m-1]$.

\medskip

When $S$ has type II, the condition $x_1+\ldots+x_{m-\epsilon}=m-1$ with $x_i\in [1,m-1]$ forces $\max x_i\leq (m-1)-(m-\epsilon-1)=\epsilon$. Thus we always have $x_i\leq \epsilon$. In particular, if $\epsilon=1$, then $x_i=1$ for all $i\in [1,m-1]$.

\medskip

When $S$ has type II, then $s\in [1,n-1]$ forces $n\geq 2$. Suppose $n=2$. Then $s=1$ and $\ord(f_1)=\ord(f_2)=2m$. Let $f'_1\in G$ be such that $\{f'_1,\,f_2\}$ is a basis for $G$ with $\ord(f'_1)=m$. Let $g=xf'_1+yf_2\in G$ with $x,\,y\in \Z$. If $y$ is odd, then $mg=xmf'_1+ymf_2=ymf_2\neq 0$, implying $\ord(g)>m$ and thus $\ord(g)=2m$. On the other hand, if $\ord(g)=2m$, then $0\neq mg=ymf_2$, implying $y$ is odd. Consequently, the elements $g\in G$ with $\ord(g)=2m$ are precisely those  $g=xf'_1+yf_2$ with $x,\,y\in\Z$ and $y$ odd, meaning any $g\in G$ with $\ord(g)=2m$ has $mg=mf_2$. In particular, $mf_1=mf_2$. This explains why both conditions of (d) always hold when $n=2$.

\medskip

If $n>1$, then there are at most $m-1$ terms of order $m$  in $S$. Indeed, if $S$ has type I(a), then all terms of order $m$ are contained in $\prod_{i=1}^{m}
      (x_{i}e_1+e_2)$. However, since $m\Sum{i=1}{m}
      (x_{i}e_1+e_2)=me_1\neq 0$, they cannot all have order $m$, meaning there are at most $m-1$ such terms. If $S$ has type I(b), then it is clear that all terms of the form $x_{i}e_1+e_2$ have order $mn>m$, leaving at most $m-1$ of order $m$, all equal to $e_1$. Finally, if $S$ has type II, then we have $\ord(f_1)\geq 2m$ as remarked above. Thus only terms contained in $\prod_{i=1}^{m-\epsilon} ( -x_{i} f_1 +
      f_2)$ can have order $m$, meaning there are at most $m-\epsilon\leq m-1$ such terms

Moreover, if $S$ has type II and contains precisely $m-1$ terms of order $m$, then we must have $\epsilon=1$ with each term from $\prod_{i=1}^{m-1} ( -x_{i} f_1 +
      f_2)$ having order $m$. However, since we have  $x_i\in [1,\epsilon]$ as remarked above, this is only possible if \[
S=f_1^{sm-1}f_2^{(n-s)m+1}(f_2-f_1)^{m-1} \quad\mbox{ with } \quad  \ord(f_2-f_1)=m \,.
\] In such case,  $\{f_2,f_2-f_1\}$ is also a generating set for $G$ with $\ord(f_2)=mn$ and $\ord(f_2-f_1)=m$, which forces $\{f_2,f_2-f_1\}$ to be a basis for $G$. Thus $S$ has type I(b) (taking $e_1=f_2-f_1$ and $e_2=f_2$).

\medskip

In particular, if $S$ had type II with  $mf_1=mf_2$ and $\epsilon=1$, then $S$ would also have type I(b). Indeed  $\epsilon=1$ forces $x_i=1$ for all $i\in [1,m-1]$ in view of $x_i\in [1,\epsilon]$, while each $-x_if_1+f_2=-f_1+f_2$ has order $m$ in view of $mf_1=mf_2$ (and the fact that every term of $S$ has its order being a multiple of $m$). Thus we would have  $m-1$ elements of order $m$, so that the above argument shows that $S$ has type I(b). This argument is what allows us to assume (e) in Main Proposition \ref{5.4}.  In particular, if $S$ has type II and $n=2$, then $\epsilon\geq 2$ and $m\geq 3$ (as $\epsilon\in [2,m-1]$) .

 \medskip

\medskip

The following lemma regarding type II sequences will be needed in the proof.

\begin{lemma} \label{lem-tech}
Let  $G = C_{m} \oplus C_{mn}$  with $n\geq 1$  and $m \ge 2$. Suppose  $S$ is a minimal zero-sum
sequence
over $G$ of length $\mathsf D (G) = m+mn-1$ that is of type II, say   \[
      S = f_1^{sm - 1} f_2^{(n-s)m+\epsilon}\prod_{i=1}^{m-\epsilon} ( -x_{i} f_1 +
      f_2)
      \] with all notation as in Main Proposition \ref{5.4}. Suppose $T\mid S$ is a subsequence with $|T|\geq 2m-1$. Then $T$ contains a subsequence $T_1\mid T$ with $\sigma(T_1)=mf_2$. Furthermore, if $T$ has no proper subsequence with this property, then $T=f_1^{m-1}f_2^{\epsilon}\prod_{i=1}^{m-\epsilon}(-x_if_1+f_2).$  \end{lemma}

\begin{proof} Since $s\in [1,n-1]$, we conclude that $n\geq 2$.
 If $s=1$, then $\vp_{f_1}(T)\leq \vp_{f_1}(S)=m-1$. On the other hand, if $s>1$, then $mf_1=mf_2$, in which case we must also have  $\vp_{f_1}(T)\leq m-1$ else $f_1^m\mid T$ will be a proper subsequence whose sum is $mf_1=mf_2$, as desired. Thus we may assume \be\label{wha1}\vp_{f_1}(T)=m-1-t\quad\mbox{ for some $t\in [0,m-1]$}.\ee
Likewise, we must have $\vp_{f_2}(T)\leq m-1$ else $f_2^m\mid T$ will be a proper subsequence whose sum is $mf_2$, as desired. By re-indexing the $-x_if_1+f_2$ appropriately, we may w.l.o.g. assume \be\label{wha2}\prod_{i=1}^\ell(-x_if_1+f_2)=\gcd\left(\prod_{i=1}^{m-\epsilon}(-x_if_1+f_2), T\right),\quad\mbox{ where $\ell\in [0,m-\epsilon]$}.\ee Hence, from the hypothesis $|T|\geq 2m-1$, we deduce that  \be\label{whalt}\vp_{f_2}(T)=|T|-\vp_{f_1}(T)-\ell\geq  m+t-\ell.\ee In particular, $\vp_{f_2}(T)\leq m-1$ forces $\ell\geq t+1\geq 1$.

Recall that $x_1+\ldots+x_{m-\epsilon}=m-1$ with $x_i\in [1,m-1]$ for all $i$. Thus $$x_1+\ldots+x_\ell=m-1-x\quad\mbox{ with } \quad x:=\Sum{i=\ell+1}{m-\epsilon}x_i\geq m-\epsilon-\ell\geq 0.$$
Consequently, if $t\leq x$, then the sequence $f_1^{m-1-t}\prod_{i=1}^\ell (-x_if_1+f_2)$ contains at least $\ell$ disjoint subsequences each having sum $f_2$ and containing precisely one term of the form $-x_if_1+f_2$, while if $t\geq x$, then the sequence $f_1^{m-1-t}\prod_{i=1}^\ell (-x_if_1+f_2)$ contains at least $\ell-\Big((m-1-x)-(m-1-t)\Big)=\ell-t+x$ disjoint subsequences each having sum $f_2$ and containing precisely one term of the form $-x_if_1+f_2$. In either case, we have  $$R_1\cdot\ldots\cdot R_{w}\mid
f_1^{m-1-t}\prod_{i=1}^\ell (-x_if_1+f_2)\quad\mbox{ with } \quad \sigma(R_i)=f_2\;\mbox { for $i\in [1,w]$},$$ where $w=\min\{\ell,\ell-t+x\}$. Moreover,  the subsequence  $R_1\cdot\ldots \cdot R_{w}$ of $f_1^{m-1-t}\prod_{i=1}^\ell (-x_if_1+f_2)$ will be proper unless $m-1-t=x_1+\ldots+x_\ell=m-1-x$, i.e., unless $t=x$.

Now, if $t<x$, then $T_1=R_1\cdot\ldots\cdot R_\ell f_2^{m-\ell}$ is a proper subsequence of $T$ (in view of \eqref{wha1}, \eqref{wha2}, \eqref{whalt} and $t\neq x$) with sum $\sigma(T_1)=mf_2$, as desired.
On the other hand, if $t\geq x$, then $T_1=R_1\cdot\ldots\cdot R_{\ell-t} f_2^{m-\ell+t}$  is a subsequence of $T$ (in view of \eqref{wha1}, \eqref{wha2} and \eqref{whalt}) with sum $\sigma(T_1)=mf_2$. Moreover, it will be a proper subsequence of $T$ unless $t=x=0$ and equality holds in \eqref{whalt}. From $x_1+\ldots +x_\ell=m-1-x=m-1$, we deduce that $\ell=m-\epsilon$ in this case (recall that $x_1+\ldots+x_{m-\epsilon}=m-1$ with $x_i\in [1,m-1]$ for all $i$), and now $$T=f_1^{m-1-t}f_2^{m+t-\ell}\prod_{i=1}^{\ell}
(-x_if_1+f_2)=f_1^{m-1}f_2^{\epsilon}\prod_{i=1}^{m-\epsilon}
(-x_if_1+f_2),$$ completing the proof.
\end{proof}

We are now ready to proceed with the proof of Theorem \ref{5.1}.

\begin{proof}[Proof of Theorem \ref{5.1}]
By Proposition \ref{2.1}, we have $\rho_3 (H) = \rho_3 (G)$. We study $\rho_3 (G)$ and recall that $\mathsf D (G) =  \mathsf D^* (G)  = m + mn  - 1$. If $n=1$, then $G = C_m \oplus C_m$,  and the theorem follows from Corollary \ref{4.2}. If $m=n=2$, then $G = C_2 \oplus C_4$, and the theorem follows from
Corollary \ref{5.2}.1. We now assume $n\geq 2$ with $m\geq 3$ when $n=2$. In particular, $\mathsf D(G)\geq 7$. It remains to show $\rho_3(G)<\rho := \lfloor 3 \mathsf D (G)/2 \rfloor=\lfloor\frac{3m+3mn-3}{2}\rfloor$ in this case. Assume by contradiction that there are $U_1, U_2, U_3, V_1, \ldots, V_{\rho} \in  \mathcal A (G)$ such that
\[
U_1U_2U_3 = V_1 \cdot \ldots \cdot V_{\rho} \,.
\] Without loss of generality, we may assume $|U_1|\geq |U_2|\geq |U_3|$ and $|V_1|\geq \ldots\geq |V_\rho|$. We continue by showing we can assume  the following assertion holds true. Note that $|U_3|=\mathsf D(G)-1$ is only possible in Assertion A if $\mathsf D(G)$ is odd and $|V_1|=2$.

\subsection*{Assertion A} $|U_1|=|U_2|=\mathsf D(G)$ and $\mathsf D(G)-1\leq |U_3|\leq \mathsf D(G)$ with the $U_i$ satisfying either
\begin{align} \nn &U_1=AB, &&-U_2=AC,  &&U_3=(-B)C,
   && |A|= \left\lceil \frac{\mathsf D(G)}{2}\right\rceil
    &&\und
    \quad |V_1|=2\quad \mbox{ or}\\  &U_1=ABw_1, &&-U_2=ACw_2, &&U_3=(-B)C(w_2-w_1),&&  |A|=\frac{\mathsf D(G)-1}{2}&&\und\quad |V_1|=3,\quad\mbox{ where}\nn\end{align}
    \begin{align}
&A=\gcd(U_1, -U_2), &&B=\gcd(U_1, -U_3), &&C=\gcd(-U_2, U_3), && |B|=|C|=\left\lfloor\frac{\mathsf D(G)}{2}\right \rfloor\und\;w_1,\,w_2\in G.\nn\end{align}

\smallskip

\begin{proof}[Proof of  Assertion A]
We trivially have $|U_1U_2U_3|=|U_1|+|U_2|+|U_3|\leq 3\mathsf D(G)$. Also,  $|V_i|\geq 2$ for each $i$ (as $0$ cannot divide any $U_i$, else $\rho\leq \mathsf D(G)+1$), implying $$3\mathsf D(G)\geq |U_1U_2U_3|=|V_1\cdot\ldots \cdot V_\rho|=\Sum{i=1}{\rho}|V_i|\geq 2\rho=2\lfloor 3 \mathsf D (G)/2 \rfloor\geq 3\mathsf D(G)-1,$$ with equality in the latter estimate only possible when $\mathsf D(G)$ is odd. It follows that, if $\mathsf D(G)$ is even, then  $|U_1|=|U_2|=|U_3|=\mathsf D(G)$ with $|V_i|=2$ for all $i$, while  if $\mathsf D(G)$ is odd, then  either  $|U_1|=|U_2|=|U_3|=\mathsf D(G)$ with $|V_1|=3$ and $|V_i|=2$ for all $i\geq 2$ or else  $|U_1|=|U_2|=\mathsf D(G)$ and $|U_3|=\mathsf D(G)-1$ with $|V_i|=2$ for all $i$.

When $|V_i|=2$ for all $i$, then $S=U_1U_2U_3$ has a factorization into length $2$ atoms. Thus $U_1=AB$, $-U_2=AC$ and $U_3=(-B)C$ for some $A,\,B,\,C\in \Fc(G)$. Since $|A|+|B|=|U_1|=\mathsf D(G)=|U_2|=|A|+|C|$, it follows that $|B|=|C|$. But now $2|B|=|B|+|C|=|U_2|\in \mathsf \{\mathsf D(G),\,\mathsf D(G)-1\}$, implying $|B|=|C|=\left\lfloor\frac{\mathsf D(G)}{2}\right \rfloor$ and $|A|=|U_1|-|B|=\mathsf D(G)-\left\lfloor\frac{\mathsf D(G)}{2}\right \rfloor=\left\lceil \frac{\mathsf D(G)}{2}\right\rceil$. If there is some $g\in \supp(B)\cap \supp(C)$, then $U_3$ will contain both $g$ and $-g$. However, since $U_3$ is an atom, this is only possible if $|U_3|=2$, contradicting that $|U_3|\geq \mathsf D(G)-1\geq 6$. Therefore we instead conclude that $\supp(B)\cap \supp(C)=\emptyset$, implying $\gcd(U_1,-U_2)=A$. Similar arguments show that $B=\gcd(U_1, -U_3)$ and $C=\gcd(-U_2, U_3)$, completing the proof of Assertion A in this case.
It remains to consider the case when $|V_1|=3$ with $|V_i|=2$ for $i\geq 2$, which is only possible when $|U_1|=|U_2|=|U_3|=\mathsf D(G)$ is odd.

If some $U_i$, say w.l.o.g. $U_3$, contains two terms from $V_1$, say $g_1g_2\mid \gcd(V_1,U_3)$,  then replacing $U_3$ by $U'_3=U_3(g_1g_2)^{-1}(g_1+g_2)$  and replacing $V_1$ by $V'_1=V_1(g_1g_2)^{-1}(g_1+g_2)$ yields atoms $U_1,\,U_2,\,U'_3\in \mathcal A(G)$  having a
factorization $U_1U_2U'_3=V'_1V_2\ldots V_\rho$ with $|U_1|=|U_2|=\mathsf D(G)$, \ $|U'_3|=\mathsf D(G)-1$ and $|V'_1|=|V_2|=\ldots=|V_\rho|=2$. These atoms also provide a counter-example to the theorem and satisfy the previously handled case of Assertion A. Thus we may assume (for the purpose of proving the theorem) that this does not occur: no length two subsequence of $V_1$ divides any $U_i$. In consequence, precisely one of each of the three terms of $V_1$ occurs in each $U_i$ while $(U_1U_2U_3)V_1^{-1}$ has a factorization into length $2$ atoms (in view of $|V_i|=2$ for $i\geq 2$). It follows that $U_1=ABw_1$, $-U_2=ACw_2$ and $U_3=(-B)C(w_2-w_1)$ for some $A,\,B,\,C\in \Fc(G)$, where $V_1=w_1(-w_2)(w_2-w_1)$.

Since $|A|+|B|+1=|U_1|=\mathsf D(G)=|U_2|=|A|+|C|+1$, it follows that $|B|=|C|$. But now $2|B|+1=|B|+|C|+1=|U_3|=\mathsf D(G)$ follows, implying $|B|=|C|=\frac{\mathsf D(G)-1}{2}=\left\lfloor\frac{\mathsf D(G)}{2}\right \rfloor$ and $|A|=|U_1|-|B|-1=\mathsf D(G)-\frac{\mathsf D(G)-1}{2}-1=\frac{\mathsf D(G)-1}{2}$.

Suppose there were some $g\in\supp(Bw_1)\cap \supp(Cw_2)$. Note $w_1\neq w_2$, else $V_1$ would contain a length $2$ zero-sum subsequence, contradicting that $V_1$ is an atom. Consequently, if $g=w_1$, then $w_1=g\in \supp(C)$, in which case $U_3$ contains the two term subsequence $w_1(w_2-w_1)$ of $V_1$, contrary to assumption. Likewise, if $g=w_2$, then $w_2\in \supp(B)$, in which case $U_3$ contains the two term subsequence $(-w_2)(w_2-w_1)$ of $V_1$, once more contrary to assumption. On the other hand, if $g\in \supp(B)\cap \supp(C)$, then   $U_3$ will contain both $g$ and $-g$,
yielding the contradiction $2=|U_3|\geq \mathsf D(G)-1=6$ as argued when $|V_i|=2$ for all $i$. So we instead conclude that $\supp(Bw_1)\cap \supp(Cw_2)=\emptyset$, implying $\gcd(U_1,-U_2)=A$. Similar arguments show that $B=\gcd(U_1, -U_3)$ and $C=\gcd(-U_2, U_3)$, completing the proof of Assertion A.
\end{proof}

In view of Assertion A, we see that we can apply Main Proposition \ref{5.4} to $U_1$ and $-U_2$ to characterize the possible structures for $U_1$ and $-U_2$. Since the roles of $U_1$ and $U_2$ are symmetric, this gives us  six cases.

\smallskip
\noindent
CASE 1: \ $U_1$ and $-U_2$ are both of type I(b), say
\[
U_1 = e_1^{m-1} \prod_{i=1}^{mn}
      (x_{i}e_1+e_2) \quad \und\quad
-U_2 = f_1^{m-1} \prod_{i=1}^{mn}
      (y_{i}f_1+f_2 ),
\] where $\{e_1, \,e_2\}$ and $\{f_1,\,f_2\}$ are bases for $G$ with $\ord(e_1)=\ord(f_1)=m$ and $\ord(e_2)=\ord(f_2)=mn>n$.

\medskip

Let $H=\la e_1,\,f_1\ra$. Since $\ord(e_1)=\ord(f_1)=m$, we conclude that $H$ is isomorphic to a subgroup of $C_m^2$. In particular, $\mathsf D(H)\leq \mathsf D(C_m^2)=2m-1$. Since $m,\,n\geq 2$ with $n\geq 3$ when $m=2$, we have $|B|=|C|\geq \frac{\mathsf D(G)-1}{2}=\frac{mn+m-2}{2}> m$. Likewise $|A|\geq \frac{\mathsf D(G)-1}{2}>m$.
Any element of the form $xe_1+e_2$ or $yf_1+f_2$, where $x,\,y\in \Z$, has order $mn>m=\ord(e_1)=\ord(f_1)$ and thus cannot be equal to $e_1$ nor $f_1$. Since $|A|\geq m+1$, we conclude that $A$ must contain a term from $U_1$ of the form $xe_1+e_2$, which must, by the previously mentioned order restriction, be equal to a term from $-U_2$ of the form $yf_1+f_2$. Hence $f_2-e_2\in H$. But now it is clear that difference between any two terms of the form $x'e_1+e_2$ and $y'f_1+f_2$, where $x',\,y'\in \Z$, must also be an element from $H$.

\smallskip

If $e_1=f_1$, then $H\cong C_m$ and $\mathsf D(H)=\mathsf D(C_m)=m$. In this case, $B=b_1\cdot\ldots\cdot b_\ell$ consists entirely of terms of the form $xe_1+e_2$ while $C=c_1\cdot \ldots\cdot c_\ell$ consists entirely of terms of the form $yf_1+f_2$, where $\ell=|B|=|C|\geq m+1$. Consequently, $(-b_1+c_1)\cdot\ldots\cdot (-b_m+c_m)\in \Fc(H)$ is a sequence of $m=\mathsf D(H)$ terms from $H$, meaning  $(-B)C$ contains a nontrivial zero-sum subsequence of length at most $2m<2\ell=|B|+|C|\leq |U_3|$. But this contradicts that $U_3$ is an atom with $(-B)C\mid U_3$. Therefore we may now assume $e_1\neq f_1$.

\smallskip

In view of $e_1\neq f_1$ and the previously mentioned order restriction, neither $e_1$ nor $f_1$ can be a term from $A$. Thus every term equal to $e_1$ in $U_1$ must be contained in $B$ except possibly one such term equal to $w_1$. Likewise, every term equal to $f_1$ in $-U_2$ must be contained in $C$ except possibly one such term equal to $w_2$. It follows that
$m-2\leq \vp_{e_1}(B)\leq m-1$ and $m-2\leq \vp_{f_1}(C)\leq m-1$. Consequently, in view of $|B|=|C|\geq m+1$,  there must be subsequences   $b_1\cdot b_2\mid B$ and $c_1\cdot c_2\mid C$ with each term $b_i$ of the form $b_i=x'_ie_1+e_2$ and each term $c_i$ of the form $c_i=y'_if_1+f_2$. Moreover, if $\vp_{e_1}(B)=\vp_{f_1}(C)=m-2$, then there exists a third term $b_3$ from $B$ also of the form $b_3=x'_3e_1+e_2$ and a third term $c_3$ from $C$ also of the form $c_3=y'_3f_1+f_2$ so that $b_1\cdot b_2\cdot b_3\mid B$ and $c_1\cdot c_2\cdot c_3\mid C$. Observe that $\vp_{f_1}(C)<m-1$, as well as  $\vp_{e_1}(B)< m-1$, is  only possible if $U_3=(-B)C(w_2-w_1)$.

\smallskip

If $\vp_{e_1}(B)=\vp_{f_1}(C)=m-1$, then $(-e_1)^{m-1}f_1^{m-1}(-b_1+c_1)\in \Fc(H)$ is a sequence of terms from $H$ of length $2m-1\geq \mathsf D(H)$, meaning  $(-B)C$ contains a nontrivial zero-sum subsequence of length at most $2m<2\ell=|B|+|C|\leq |U_3|$. But this contradicts that $U_3$ is an atom with $(-B)C\mid U_3$.

If $\vp_{e_1}(B)=\vp_{f_1}(C)=m-2$, then $(-e_1)^{m-2}f_1^{m-2}(-b_1+c_1)(-b_2+c_2)(-b_3+c_3)\in \Fc(H)$ is a sequence  of length $2m-1\geq \mathsf D(H)$, meaning  $(-B)C$ contains a nontrivial zero-sum subsequence, contradicting that $U_3$ is an atom since $U_3=(-B)C(w_2-w_1)$.

If $\vp_{e_1}(B)=m-1$ and $\vp_{f_1}(C)=m-2$, then $(-e_1)^{m-1}f_1^{m-2}(-b_1+c_1)(-b_2+c_2)\in \Fc(H)$ is a sequence of length $2m-1\geq \mathsf D(H)$, meaning  $(-B)C$ contains a nontrivial zero-sum subsequence, contradicting that $U_3$ is an atom since $U_3=(-B)C(w_2-w_1)$.

If $\vp_{e_1}(B)=m-2$ and $\vp_{f_1}(C)=m-1$, then $(-e_1)^{m-2}f_1^{m-1}(-b_1+c_1)(-b_2+c_2)\in \Fc(H)$ is a sequence of length $2m-1\geq \mathsf D(H)$, meaning  $(-B)C$ contains a nontrivial zero-sum subsequence, contradicting that $U_3$ is an atom since $U_3=(-B)C(w_2-w_1)$, which completes CASE 1.

\smallskip
\noindent
CASE 2: \ $U_1$ and $-U_2$ are both of type I(a), say
\[
U_1 = e_1^{mn-1} \prod_{i=1}^{m}
      (x_{i}e_1+e_2) \quad \und \quad
-U_2 = f_1^{mn-1} \prod_{i=1}^{m}
      (y_{i}f_1+f_2) \,.
\]where $\{e_1, \,e_2\}$ and $\{f_1,\,f_2\}$ are bases for $G$ with $\ord(e_1)=\ord(f_1)=mn>m$ and $\ord(e_2)=\ord(f_2)=m$.

\medskip

Since $m,\,n\geq 2$ with $n\geq 3$ when $m=2$, we have $mn-1>\frac{mn+m}{2}=\frac{\mathsf D(G)+1}{2}\geq |A|$. If $e_1=f_1$, then $\gcd(U_1,-U_2)=A$ implies $|A|\geq \vp_{e_1}(U_1)=mn-1$, contrary to what we just noted. Therefore $e_1\neq f_1$. On the other hand, since $\vp_{e_1}(U_1)=\vp_{f_1}(-U_2)=mn-1>\frac{\mathsf D(G)+1}{2}\geq \mathsf D(G)-|A|=|U_1|-|A|=|U_2|-|A|$, we must have $e_1,\,f_1\in \supp(A)$. It follows that $$e_1=yf_1+f_2\quad\und\quad f_1=xe_1+e_2\quad\mbox{ for some $x,\,y\in \Z$}.$$ Since $U_1$ contains at most $m$ terms not equal to $e_1$, we deduce that $\vp_{e_1}(A)\geq |A|-m\geq \frac{\mathsf D(G)-1}{2}-m=\frac12 mn-\frac{m}{2}-1$. However, since $e_1\neq f_1$ with the highest multiplicity of a term in $-U_2$ other than $f_1$ being $m-1$, we have $$\vp_{e_1}(A)\leq \vp_{yf_1+f_2}(-U_2)\leq m-1.$$ Hence $\frac12 mn-\frac{m}{2}-1\leq \vp_{e_1}(A)\leq m-1$, implying $n\leq 3$.

\bigskip

Suppose $n=3$. Then $\mathsf D(G)=4m-1$ and equality must hold in all estimates used to derive $n\leq 3$ above. In particular, $|A|=\frac{\mathsf D(G)-1}{2}$, forcing the case corresponding to $|V_1|=3$ in Assertion A, and  all $m$ terms of $U_1$ not equal to $e_1$ must be contained in $A$.   Arguing as in the previous paragraph, we must also have $$\frac12 mn-\frac{m}{2}-1\leq |A|-m\leq \vp_{f_1}(A)\leq \vp_{xe_1+e_2}(U_1)\leq m-1,$$ implying $n\leq 3$. Once more, equality must hold in all these estimates, meaning all $m$ terms of $-U_2$ not equal to $f_1$ must be contained in $A$. Consequently,
$$U_3=(-B)C(w_2-w_1)=(-e_1)^{2m-1}f_1^{2m-1}(f_1-e_1)=(-e_1)^{2m-1}(xe_1+e_2)^{2m-1}((x-1)e_1+e_2).$$
Since $\sigma(U_3)=0$, we see that $x\equiv 1\mod 3$, and now it is easily noted that $(-e_1)^m(xe_1+e_2)^m$ is a proper zero-sum subsequence of $U_3$, contradicting that $U_3$ is an atom. So we may instead assume $n=2$.

\bigskip

Since $n=2$, it follows that  $\mathsf D(G)=3m-1$ and $m\geq 3$.
Recall that $e_1=yf_1+f_2$ and $f_1=xe_1+e_2$. Thus, since $\ord(e_1)=\ord(f_1)=2m$, we conclude that $x$ and $y$ are both odd, whence
\be\label{goweer}
me_1=myf_1=mf_1=mxe_1\quad\mbox{ with }\quad \ord(me_1)=\ord(mf_1)=2.
\ee
If $\vp_{-e_1}(-B)\geq m$ and $\vp_{f_1}(C)\geq m$, then $(-e_1)^mf_1^m$ is a zero-sum subsequence of $U_3$ (in view of \eqref{goweer}) of length $2m<3m-2=\mathsf D(G)-1\leq |U_3|$, contradicting that $U_3$ is an atom. Therefore we may assume either $\vp_{-e_1}(-B)< m$ or $\vp_{f_1}(C)< m$, say w.l.o.g. $\vp_{-e_1}(-B)< m$ (the role of $e_1$ in $U_1$ is identical to that of $f_1$ in $-U_2$).

As noted earlier, $\vp_{e_1}(A)\leq m-1$. Consequently, if $|V_1|=2$, then $\vp_{-e_1}(-B)=\vp_{e_1}(U_1)-\vp_{e_1}(A)\geq 2m-1-(m-1)=m$, contrary to our assumption above. Thus we must have $|V_1|=3$,  which is only possible (in view of Assertion A) if $|U_3|=\mathsf D(G)=3m-1$ is odd. Thus $2\mid m$ and $m\geq 4$.

Applying the above argument when $|V_1|=3$, we again obtain the contradiction $\vp_{-e_1}(U_3)\geq m$ unless $\vp_{e_1}(A)=m-1$ and $w_1=e_1$. It follows that there are at most $|A|-\vp_{e_1}(A)=\frac{m}{2}$ terms of $A$ not equal to $e_1$. Hence, since $f_1\neq e_1$, we conclude that $\vp_{f_1}(A)\leq \frac{m}{2}$, implying \be\label{mickey1}\vp_{f_1}(C)\geq 2m-1-\frac{m}{2}-1=\frac32 m-2,\ee with equality only possible if $w_2=f_1$ and $w_2-w_1=f_1-e_1=(x-1)e_1+e_2$. Since $\vp_{e_1}(A)=m-1$ and $w_1=e_1$, we have   $$-B=(-e_1)^{m-1}\prod_{i=1}^{m/2}(-x_ie_1-e_2),$$ where we have  appropriately re-indexed the terms $x_ie_1+e_2$ in $U_1$ so that the first $\frac{m}{2}$ terms correspond to those from $B$. Thus $$U_3=(-e_1)^{m-1} \left(\prod_{i=1}^{m/2}(-x_ie_1-e_2)\right) f_1^{\frac32 m-2} g_1 g_2=(-e_1)^{m-1} \left(\prod_{i=1}^{m/2}(-x_ie_1-e_2)\right) (xe_1+e_2)^{\frac32 m-2} g_1g_2$$
 with w.l.o.g. $g_1\in \{f_1,y_1f_1+f_2\}$ (by re-indexing the $y_if_1+f_2$ appropriately) and $g_2=w_2-w_1=w_2-e_1$.

\medskip

If $g_1=f_1$, let $g=g_1=f_1=xe_1+e_2$. If $g_1\neq f_1$, the equality must hold in \eqref{mickey1}. In this case, let $g=g_2=f_1-e_1=(x-1)e_1+e_2$. Regardless, we see that $g=g_j= ze_1+e_2$ for some $z\in \{x,\,x-1\}$ and $j\in [1,2]$. To avoid a zero-sum subsequence of $$(-e_1)^{m-1}(-x_1e_1-e_2)(ze_1+e_2),$$ which would contradict that $U_3$ is an atom,  we must have $x_1\notin \{z,z-1,\ldots,z-(m-1)\}$ modulo $2m$. On the other hand, in view of \eqref{goweer}, we have $\sigma((xe_1+e_2)^{m})=me_1$, so that to avoid a  zero-sum subsequence of $$(-e_1)^{m-1}(xe_1+e_2)^{m}(-x_1e_1-e_2)(ze_1+e_2),$$ which would contradict that $U_3$ is an atom in view of $\frac32m-2\geq m$,  we must have $x_1\notin \{m+z,m+z-1,\ldots,m+z-(m-1)\}$ modulo $2m$. However, this leaves no possibilities left for the value of $x_1$ modulo $2m$, which is a contradiction that concludes CASE 2.

\smallskip
\noindent
CASE 3: \ Either $U_1$ is of type I(b) and $-U_2$ is of type I(a) or else $U_1$ is of type I(a) and $-U_2$ is of type I(b), say w.l.o.g. the former with
\[
U_1 = e_1^{m-1} \prod_{i=1}^{mn}
      (x_{i}e_1+e_2) \quad \und \quad
-U_2 = f_1^{mn-1} \prod_{i=1}^{m}
      (y_{i}f_1+f_2),
\] where $\{e_1,\,e_2\}$ and $\{f_1,\,f_2\}$ are bases of $G$ with $\ord(e_1)=\ord(f_2)=m$ and $\ord(e_2)=\ord(f_1)=mn>m$.

\medskip

Since $m,\,n\geq 2$ with $n\geq 3$ when $m=2$, we have $\vp_{f_1}(-U_2)=mn-1>\frac{\mathsf D(G)+1}{2}\geq \mathsf D(G)-|A|=|U_2|-|A|$, implying $f_1\in \supp(A)$. Consequently, since $f_1$ cannot equal $e_1$ due to $\ord(f_1)=mn>m=\ord(e_1)$, it follows that $$f_1=xe_1+e_2\quad\mbox{ for some $x\in \Z$}.$$ Let \be\label{stark1}y=\vp_{e_1}(B)\in [0,m-1].\ee Then $\vp_{e_1}(A)=m-1-y-\epsilon$, where $\epsilon =1$ if $|V_1|=3$ and $w_1=e_1$, and $\epsilon =0$ otherwise. Since $f_1\neq e_1$, it follows that $\vp_{f_1}(A)\leq |A|-\vp_{e_1}(A)= |A|-m+1+y+\epsilon$, implying \be\label{mickey2}\vp_{f_1}(C)\geq mn-1-\delta-|A|+m-1-y-\epsilon\geq \frac12 mn+\frac12 m-3-y,\ee
where $\delta=1$ if $|V_1|=3$ and $w_2=f_1$, and $\delta =0$ otherwise.
Moreover, the estimate on the far right of \eqref{mickey2} improves by $1$ unless $w_1=e_1$ and $w_2=f_2$, in which case $w_2-w_1=(x-1)e_1+e_2$ is a term of $U_3$. As a result, we see that $U_3(-B)^{-1}$ contains at least $\frac12 mn+\frac12 m-2-y$ terms from $e_2+\la e_1\ra$, say $c_1\cdot\ldots\cdot c_s\mid U_3(-B)^{-1}$ with
\be\label{stark2}s\geq \frac12 mn+\frac12 m-2-y\quad\und\quad c_i\in e_2+\la e_1\ra\;\mbox{ for all $i$}.\ee On the other hand, per definition of $y$, we see that $-B$ contains $|B|-y\geq \frac12 mn+\frac12 m-1-y$ terms from $-e_2+\la e_1\ra$, say $b_1\cdot\ldots\cdot b_t\mid -B$ with \be\label{stark3}t\geq \frac12 mn+\frac12 m-1-y\quad \und\quad b_i\in -e_2+\la e_1\ra\;\mbox{ for all $i$}.\ee

Now $e_1\in \la e_1\ra$ and $b_i+c_i\in \la e_1\ra$ for all $i\in [1,\min\{s,t\}]$, while $\mathsf D(\la e_1\ra)=\mathsf D(C_m)=m$. Moreover, $(\frac12 mn+\frac12 m-2-y)+y> m-1$ in view of $m,\,n\geq 2$ with $n\geq 3$ when $m=2$. Consequently, we conclude from \eqref{stark1}, \eqref{stark2} and \eqref{stark3} that $U_3$ contains a nontrivial zero-sum subsequence of length at most $2\lceil \frac12 mn+\frac12 m-2-y\rceil+y\leq mn+m-3<\mathsf D(G)-1\leq |U_3|$, contradicting that $U_3$ is an atom.

\smallskip
\noindent
CASE 4: \  $U_1$ and $-U_2$ are both of type II, say
\[
U_1
= f_1^{s_1m-1}f_2^{(n-s_1)m+\epsilon_1} \prod_{i=1}^{m-\epsilon_1}
      (-y_{i}f_1+f_2) \quad \und \quad -U_2=g_1^{s_2m-1}g_2^{(n-s_2)m+\epsilon_2}\prod_{i=1}^{m-\epsilon_2}(-z_ig_1+g_2),
\] where $\{f_1,\,f_2\}$ and $\{g_1,\,g_2\}$ are generating sets for $G$ such that $\ord(f_2)=\ord(g_2)=mn>m$ and $\ord(f_1),\,\ord(g_1)\geq 2m$, where $s_1,\,s_2\in [1,n-1]$, \ $\epsilon_1,\,\epsilon_2\in [1,m-1]$  and $y_i,\,z_i\in [1,m-1]$ for all $i$, and where $y_1+\ldots+y_{m-\epsilon_1}=z_1+\ldots+z_{m-\epsilon_2}=m-1$. Moreover, either $s_1=1$ or $mf_1=mf_2$ and either $s_2=1$ or $mg_1=mg_2$.

\medskip

Per the remarks after Main Proposition \ref{5.4}, let $\{f'_1,\,f_2\}$ and $\{g'_1,\,g_2\}$ be bases for $G$ with $\ord(f'_1)=\ord(g'_1)=m$ such that $$f_1=f'_1+\alpha f_2\quad\und\quad g_1=g'_1+\beta g_2$$ for some $\alpha,\,\beta\in \Z$. We distinguish two subcases.

\smallskip
\noindent
CASE 4.1: \ $n\geq 3$.

\medskip

Since $n\geq 3$, we have $|A|\geq |B|=|C|\geq \frac{mn+m-2}{2}\geq 2m-1$. Thus   $\vp_{\{f_1,\,f_2\}}(A)\geq |A|-(m-\epsilon_1)\geq |A|-m+1\geq m>m-\epsilon_2$, implying  \be\label{nice-intersect}\{f_1, f_2\}\cap\{g_1, g_2\}\ne\emptyset.\ee Also,
applying Lemma \ref{lem-tech} to $B\mid U_1$ and $C\mid -U_2$, we conclude that there exist subsequences $T_1\mid B$ and $T_2\mid C$ with $\sigma(T_1)=mf_2$, \ $\sigma(T_2)=mg_2$ and $|T_1|,\,|T_2|\leq 2m-1$.

\medskip

Suppose $mf_2=mg_2$, so that $\sigma(T_1)=\sigma(T_2)$. Then $(-T_1)T_2$ is a zero-sum subsequence of $(-B)C\mid U_3$, which contradicts that $U_3$ is an atom unless $(-T_1)T_2=(-B)C=U_3$ with $|B|=|C|=|T_1|=|T_2|=2m-1$, implying $n=3$. However, in view of the equality conditions in Lemma \ref{lem-tech}, this is only possible if
$$U_3=(-B)C=(-f_1)^{m-1}(-f_2)^{\epsilon_1}\prod_{i=1}^{m-\epsilon_1}(y_if_1-f_2)\cdot g_1^{m-1}g_2^{\epsilon_2}
\prod_{i=1}^{m-\epsilon_2}(-z_ig_1+g_2).$$ In particular, the terms $-f_1$, $-f_2$, $g_1$ and $g_2$ all occur in $U_3$ in view of $m\geq 2$ and $\epsilon_1,\,\epsilon_2\geq 1$. But then \eqref{nice-intersect} ensures that $U_3$ contains a zero-sum subsequence of length $2$, contradicting that $U_3$ is an atom with $|U_3|=4m-2>2$. So we instead conclude that \be\label{yeeyo} mf_2\neq mg_2.\ee

\medskip

If $s_1>1$ and $s_2>1$, then $mf_1=mf_2$ and $mg_1=mg_2$, which combined with \eqref{nice-intersect} yields $mf_1=mf_2=mg_1=mg_2$, contrary to \eqref{yeeyo}. Therefore we may w.l.o.g. assume $$s_2=1\quad\und\quad\vp_{g_1}(-U_2)=m-1.$$ Since $|A|\geq 2m-1>\vp_{g_1}(-U_2)+m-\epsilon_2$, we conclude that $g_2\in \supp(A)$. Observe that $$g_2\neq f_2,$$ for  $g_2=f_2$ would contradict \eqref{yeeyo}. In consequence, we find that
$$g_2=f_1\quad\mbox{ or }\quad g_2=-yf_1+f_2\;\mbox{ for some $y\in [1,m-1]$}.$$ This gives two further subcases.

\smallskip
\noindent
CASE 4.1.1: \ $g_2=-yf_1+f_2$ for some $y\in [1,m-1]$.

\medskip

Now $g_2\neq f_2$ as already remarked. Also, $g_2=-yf_1+f_2\neq f_1$ as remarked after Main Proposition \ref{5.4}. Thus \eqref{nice-intersect} ensures that we must have $$g_1=f_1\quad \mbox{ or }\quad g_1=f_2.$$
If $g_1=f_1$, then  $f_1$ can have multiplicity at most  $\vp_{g_1}(-U_2)=m-1$  in $A$, meaning $f_2$  must also be contained in $A$ in view of $|A|\geq 2m-1$. By an analogous argument, if $g_1=f_2$, then $f_1$ must be contained in $A$. In other words, in both cases, we have  $$f_1,\,f_2\in \supp(A).$$

\medskip

Suppose that $g_1=f_1$. Then, as $f_2\in\supp(A)$ but $f_2\neq f_1=g_1$ and $f_2\neq g_2$, it follows that $f_2=-zg_1+g_2$ for some $z\in [1,m-1]$. Thus $$f_2=-zg_1+g_2=-zf_1+g_2=-zf_1-yf_1+f_2,$$ implying that $(z+y)f_1=0$ with $z+y\in [2,2m-2]$. However, since $\ord(f_1)\geq 2m$, this is not possible. So we instead conclude that $$g_1=f_2.$$

\medskip

Now $f_1\in \supp(A)$ but $g_1=f_2\neq f_1$ and $g_2=-yf_1+f_2\neq f_1$ as remarked after Main Proposition \ref{5.4}. In consequence, $f_1=-zg_1+g_2$ for some $z\in [1,m-1]$. Thus $$f'_1+\alpha f_2=f_1=-zg_1+g_2=-zf_2+g_2=-zf_2-yf_1+f_2=-yf'_1+(1-z-\alpha y)f_2,$$ which, in view of $y\in [1,m-1]$, is only possible if $y=m-1$ and \be\nn\alpha\equiv 1-z-\alpha y\equiv 1-z-\alpha(m-1)\mod mn.\ee  The above congruence  implies that $z\equiv 1-\alpha m\mod mn$, which, in view of $z\in [1,m-1]$, is only possible if $z=1$ and $\alpha m\equiv 0\mod mn$. Thus $mf_1=m(f'_1+\alpha f_2)=mf'_1+\alpha m f_2=0$, contradicting that $\ord(f_1)\geq 2m$ for type II.

\smallskip
\noindent
CASE 4.1.2: \ $g_2=f_1$.

\medskip

If $s_1>1$, then $mg_2=mf_1=mf_2$, contrary to \eqref{yeeyo}. Therefore $$s_1=1\quad\und\quad \vp_{f_1}(U_1)=m-1.$$  Since $|A|\geq 2m-1>\vp_{f_1}(-U_2)+m-\epsilon_1$, we conclude that $f_2\in \supp(A)$. As already remarked, we have $f_2\neq g_2$. Consequently,
$$f_2=g_1\quad\mbox{ or }\quad f_2=-zg_1+g_2\;\mbox{ for some $z\in [1,m-1]$}.$$ Observe, however, that the roles of $U_1$ and $-U_2$ are now symmetric (we have the same information about $-U_2$ that we did about $U_1$ before CASE 4.1.1). Thus, if $f_2=-zg_1+g_2$ for some $z\in [1,m-1]$, then, swapping the roles of $U_1$ and $-U_2$, we fall under the hypotheses of CASE 4.1.1, and the proof is complete by those prior arguments. So, combined with the subcase hypothesis, we may instead assume \be\label{golg}f_2=g_1\quad\und \quad f_1=g_2.\ee
Now $\vp_{f_1}(A)\leq m-1$ and $\vp_{f_2}(A)=\vp_{g_1}(A)\leq m-1$ in view of $s_1=s_2=1$. Consequently, since $|A|\geq 2m-1$, we conclude from \eqref{golg} that $-yf_1+f_2=-zg_1+g_2=-zf_2+f_1$ for some $y,\,z\in [1,m-1]$. Thus $$0=(1+y)f_1-(1+z)f_2=(1+y)f'_1+(\alpha (1+y)-1-z)f_2,$$ which, in view of $y\in [1,m-1]$, is only possible if $y=m-1$ and \be\nn 0\equiv\alpha(1+y)-1-z\equiv \alpha m-1-z\mod mn.\ee The above congruence  implies that $z\equiv \alpha m-1\mod mn$, which, in view of $z\in [1,m-1]$, is only possible if $z=m-1$ and $\alpha m\equiv m\mod mn$. Thus $mg_2=mf_1=m(f'_1+\alpha f_2)=mf_2$, contradicting \eqref{yeeyo} and completing CASE 4.1.

\smallskip
\noindent
CASE 4.2: \ $n=2$.

\medskip

Since $s_1,\,s_2\in [1,n-1]=[1,1]$, we conclude that $s_1=s_2=1$. We also have $m\geq 3$ and $\epsilon_1,\,\epsilon_2\geq 2$ in view of $n=2$ (the latter per (d) and (e) in Main Proposition \ref{5.4}). Now
\[
U_1=f_1^{m-1}f_2^{m+\epsilon_1}\prod_{i=1}^{m-\epsilon_1}(-y_{i}f_1+f_2) \quad \und \quad  -U_2=g_1^{m-1}g_2^{m+\epsilon_2}\prod_{i=1}^{m-\epsilon_2}(-z_{i}g_1+g_2)
\]
with $\ord(f_1)=\ord(f_2)=\ord(g_1)=\ord(g_2)=2m$ and  (as remarked after Main Proposition \ref{5.4}) \be\label{m-unity}mf_1=mf_2=mg_1=mg_2.\ee
Observe that $$\frac32m-1\leq |A|\leq \frac32m\quad\und\quad\frac32m-1\leq |B|=|C|\leq\frac32m-\frac12.$$
If neither $f_2$ nor $g_2$ is a term from $A$, then $(-f_2)^mg_2^m$ will be a  subsequence of $U_3$ which is zero-sum (in view of \eqref{m-unity}) and has length $2m<3m-2\leq |U_3|$, contradicting that $U_3$ is an atom. Therefore  \be\label{f_2inSupp}f_2\in \supp(A)\quad\mbox{ or } \quad g_2\in \supp(A).\ee We handle several subcases.

\smallskip
\noindent
CASE 4.2.1: \ $f_2=g_2$.

We may w.l.o.g. assume $\epsilon_1\leq \epsilon_2$. Then $\vp_{f_2}(A)=m+\epsilon_1$ and $\vp_{f_2}(B)=0$. As remarked after Main Proposition \ref{5.4}, we have $y_i\leq \epsilon_1$ and $z_j\leq \epsilon_2$  for all $i$ and $j$. Also, since there are precisely $2m>3m-1-(\frac32m-1)\geq \mathsf D(G)-|B|$ terms of $U_1$ of the form $-xf_1+f_2$ with $x\in [0,m-1]$,  and since $\vp_{f_2}(B)=0$, it follows that
$$yf_1-f_2\in\supp(-B)\quad \mbox{ for some } y\in [1,\epsilon_1]\subseteq [1,m-1].$$ Now, since $\vp_{f_2}(B)=0$, we have $\vp_{f_1}(B)\geq |B|-(m-\epsilon_1)\geq \frac{m}{2}-1+\epsilon_1\geq \epsilon_1$. Thus $(-f_1)^y(yf_1-f_2)$ is a subsequence of $-B$. If $f_2=g_2\in \supp(C)$, then $(-f_1)^y(yf_1-f_2)f_2$ would be a zero-sum subsequence of $U_3$ of length
$y+2\leq m+1<3m-2\leq |U_3|$, contradicting that $U_3$ is an atom. Therefore we may instead assume $\vp_{g_2}(C)=0$. But now, repeating the prior arguments for $-U_2$ instead of $U_1$, we find that $$-zg_1+g_2=-zg_1+f_2\in\supp(C)\quad \mbox{ for some } z\in [1,\epsilon_2]\subseteq [1,m-1]$$ and that $\vp_{g_1}(C)\geq |C|-(m-\epsilon_2)\geq \frac{m}{2}-1+\epsilon_2\geq \epsilon_2$. Thus $g_1^z(-zg_1+f_2)(-f_1)^y(yf_1-f_2)$ is a zero-sum subsequence of $U_3$ of length $z+y+2\leq \epsilon_1+\epsilon_2+2\leq 2m<3m-2\leq |U_3|$ (in view of $m\geq 3$), contradicting that $U_3$ is an atom and completing the subcase.

\smallskip
\noindent
CASE 4.2.2: \ $f_2=g_1$ or $g_2=f_1$.

By symmetry, we may w.l.o.g. assume $$f_2=g_1.$$
Then $\vp_{f_2}(A)=\vp_{g_1}(A)=m-1$, meaning $\vp_{g_1}(C)=\vp_{f_2}(C)=0$ and $\epsilon_1+1\geq \vp_{-f_2}(-B)\geq \epsilon_1.$

\bigskip

Suppose $g_2=f_1$. Then  $\vp_{f_1}(A)=\vp_{g_2}(A)=m-1$, yielding $\vp_{g_2}(C)=\vp_{f_1}(C)\geq\epsilon_2$ and  $\frac32m\geq |A|\geq \vp_{f_2}(A)+\vp_{f_1}(A)= 2m-2$, which is only possible if $3\leq m\leq 4$ with $|A|=2m-2$ and $|V_1|=2$.
In this case, $$(-f_2)^{\epsilon_1}f_1^{\epsilon_2}\prod_{i=1}^{m-\epsilon_1}(y_if_1-f_2)
\prod_{i=1}^{m-\epsilon_2}(-z_if_2+f_1)\mid U_3.$$ Thus, if $\epsilon_1=\epsilon_2=m-1$, then
 $\prod_{i=1}^{m-\epsilon_1}(y_if_1-f_2)
\prod_{i=1}^{m-\epsilon_2}(-z_if_2+f_1)$ is a proper subsequence of $U_3$ with sum $(m-1)f_1-f_2-(m-1)f_2+f_1=mf_1-mf_2=0$ (in view of \eqref{m-unity}), contradicting that $U_3$ is an atom. Therefore we may w.l.o.g. assume  that $\epsilon_1\leq m-2$.
Hence, since $\epsilon_i\in [2,m-1]$ for $n=2$, it follows that $m=4$ with $2\leq \epsilon_1\leq m-2$, so that $\epsilon_1=2$. Consequently, since $y_1+y_2=m-1=3$ with $y_i\in [1,\epsilon_1]=[1,2]$, we see that w.l.o.g. $y_1=1$ and $y_2=2$. Likewise, if $\epsilon_2=2$, then w.l.o.g. $z_1=1$ and $z_2=2$, while if $\epsilon_2=3=m-1$, then $z_1=m-1=3$. In the former case, $(-2f_2+f_1)(f_1-f_2)(2f_1-f_2)$ is  a proper zero-sum subsequence of $U_3$ (in view of \eqref{m-unity} and $m=4$), while in the latter case, $(-3f_2+f_1)(2f_1-f_2)f_1$ is a proper zero-sum subsequence of $U_3$ (again, in view of \eqref{m-unity} and $m=4$), both contradicting that $U_3$ is an atom. So we instead conclude that $$g_2\neq f_1.$$

\bigskip

Suppose next that $f_1,\,g_2\in \supp(A)$. In view of $g_2\neq f_1$ and $g_1=f_2$, this is only possible if $$f_1=-zg_1+g_2=-zf_2+g_2\quad\und\quad g_2=-yf_1+f_2\quad\mbox{ for some $y\in [1,m-1]$ and $z\in [1,m-1]$}.$$ Thus $$f'_1+\alpha f_2=f_1=-zf_2+g_2=-zf_2-yf_1+f_2=-yf'_1+(1-z-\alpha y)f_2.$$ However, since $y\in  [1,m-1]$, this is only possible if $y=m-1$ and $1-z-\alpha y= 1-z-\alpha(m-1)\equiv \alpha\mod 2m$. Hence $z\equiv 1-\alpha m\mod 2m$, which, in view of $z\in [1,m-1]$, is only possible if $z=1$ with $\alpha m \equiv 0\mod 2m$, implying $mf_1=mf'_1+\alpha mf_2=0$. Since this contradicts that $\ord(f_1)=2m>m$, we may now assume $$f_1\notin \supp(A)\quad\mbox{ or }\quad g_2\notin \supp(A).$$

\bigskip

Suppose $f_1\notin \supp(A)$. Then $A\mid f_2^{m-1}\prod_{i=1}^{m-\epsilon_1}(-y_if_1+f_2)$. If $|\supp(A)|\geq 3$, then, since $g_1=f_2$, we must have $-yf_1+f_2=-zg_1+g_2$ and $-y'f_1+f_2=-z'g_1+g_2$ for some distinct $y,\,y'\in [1,m-1]$ and distinct $z,\,z'\in [0,m-1]$, implying $$(y-y')f'_1+\alpha(y-y')f_2=(y-y')f_1=(z-z')g_1=(z-z')f_2.$$ Thus $y\equiv y'\mod m$, which, in view of $y,\,y'\in [1,m-1]$, forces $y=y'$, contrary to assumption. Therefore we must instead have $$|\supp(A)|=2.$$ Let $-yf_1+f_2$ be the element of $\supp(A)\setminus\{f_2\}$, where $y\in [1,m-1]$. Then $-yf_1+f_2$ has  multiplicity at least $\vp_{-yf_1+f_2}(A)= |A|-\vp_{f_2}(A)=|A|-(m-1)\geq \frac{m}{2}$. Consequently, $\frac{m}{2}y\leq \vp_{-yf_1+f_2}(A)y\leq y_1+\ldots+y_{m-\epsilon_1}=m-1$, which together with  $y\in [1,m-1]$ ensures that $y=1$, so that $-f_1+f_2\in \supp(A)$. Since $-f_1+f_2\in \supp(A)$ with $g_1=f_2$, it follows that  $$-f_1+f_2=-zg_1+g_2=-zf_2+g_2\in \supp(A)\quad\mbox{ for some $z\in [0,\epsilon_2]$}.$$ If $z=0$, then $mg_2=-mf_1+mf_2=0$ (in view of \eqref{m-unity}), contradicting that $\ord(g_2)=2m$. Therefore we must have $z\in [1,m-1]$.
Then $-zg_1+g_2=-f_1+f_2$ has  multiplicity at least $\vp_{-zg_1+g_2}(A)=\vp_{-f_1+f_2}(A)=\vp_{-yf_1+f_2}(A)\geq  \frac{m}{2}$. Consequently, $\frac{m}{2}z\leq \vp_{-zg_1+g_2}(A)z\leq z_1+\ldots+z_{m-\epsilon_2}=m-1$, which together with  $z\in [1,m-1]$ ensures that $z=1$. Thus $$-f_1+f_2=-zg_1+g_2=-zf_2+g_2=-f_2+g_2.$$
Hence $g_2=-f_1+2f_2$ and $\supp(A)=\{f_2,\,-f_1+f_2\}=\{g_1,\,-g_1+g_2\}$.

Since $f_1\notin \supp(A)$, we have \be\label{mickey3}\vp_{-f_1}(-B)\geq \vp_{f_1}(U_1)-1=m-2,\ee with equality only possible if $|V_1|=3$ with $w_1=f_1$. Since $f_2=g_1$ with $\vp_{f_2}(A)=\vp_{g_1}(A)=m-1$, we have \be\nn\vp_{-f_2}(-B)\geq \vp_{f_2}(U_1)-1-\vp_{f_2}(A)=\epsilon_1\geq 2\ee (recall that $\epsilon_2\geq 2$ for $n=2$).  Since $g_2=-f_1+2f_2\notin \{-f_1+f_2,f_2=g_1\}=\supp(A)$, we have $$\vp_{-f_1+2f_2}(C)=\vp_{g_2}(C)\geq \vp_{g_2}(-U_2)-1=m+\epsilon_2-1\geq m+1.$$
Since $-yf_1+f_2=-f_1+f_2\in\supp(A)$, we know $y_k=1$ for some $k\in [1,m-\epsilon_1]$. If $y_i=1$ for all $i\in [1,m-\epsilon_1]$, then $y_1+\ldots+y_{m-\epsilon_1}=m-1$ forces $\epsilon_1=1$, contradicting that $\epsilon_1\geq 2$ for $n=2$. Therefore we may instead assume there is some $y_j\geq 2$ with $j\in [1,m-\epsilon_1]$. Then, since $y_1+\ldots+y_{m-\epsilon_1}=m-1$ with at least one $y_k=y=1$, we conclude that $2\leq y_j\leq m-2$, implying $$m\geq 4.$$ Since $\supp(A)=\{f_2,\,-f_1+f_2\}$, we must either have $y_jf_1-f_2\in \supp(-B)$ or $w_1=-y_jf_1+f_2$. In the former case,
$$(y_jf_1-f_2)(-f_1+2f_2)(-f_1)^{y_j-1}(-f_2)$$ is a zero-sum subsequence of $U_3$ of length
$y_j+2\leq m<3m-2\leq |U_3|$, contradicting that $U_3$ is an atom. In the latter case, $|V_3|=3$ and we have strict inequality in \eqref{mickey3}, in which case $$(-f_1)^{m-1}(-f_2)^2(-f_1+2f_2)^{m+1}$$ is a zero-sum subsequence of $U_3$ having length $2m+2<3m-1=|U_3|$ (in view of $m\geq 4$), contradicting that $U_3$ is an atom.
 So we may now assume $$f_1\in \supp(A)\quad\und\quad g_2\notin \supp(A).$$

\bigskip

Since $g_2\notin \supp(A)$, we have $A\mid g_1^{m-1}\prod_{i=1}^{m-\epsilon_2}(-z_ig_1+g_2)=f_2^{m-1}\prod_{i=1}^{m-\epsilon_2}(-z_if_2+g_2)$.
Thus,  since $f_1\in \supp(A)$, we have $$f_1=-xg_1+g_2=-xf_2+g_2\quad\mbox{ for some $x\in [1,m-1]$}.$$
 If $\supp(A)\neq \{f_2,\,f_1\}$, then  $-yf_1+f_2=-zg_1+g_2$ for some $y,\,z\in [1,m-1]$. In this case, $$-yf'_1+(1-\alpha y)f_2=-yf_1+f_2=-zg_1+g_2=-zf_2+(xf_2+f_1)=f'_1+(x-z+\alpha)f_2.$$ Thus, since $y\in [1,m-1]$, it follows that $y=m-1$ with $x-z\equiv 1-\alpha m\mod 2m$. Since $x-z\in [-(m-2),m-2]$ (in view of $x,\,y\in [1,m-1]$), we conclude that $x-z=1$ with $\alpha m\equiv 0\mod 2m$. But this means $mf_1=mf'_1+\alpha mf_2=0$, contradicting that $\ord(f_1)=2m$. Therefore, we instead conclude that $$\supp(A)=\{f_1,\,f_2\},$$ whence $\vp_{-xg_1+g_2}(A)=\vp_{f_1}(A)=|A|-\vp_{f_2}(A)=|A|-m+1\geq \frac{m}{2}$. Thus $f_1=-xg_1+g_2=-g_1+g_2$ in view of $\frac{m}{2}x\leq \vp_{-xg_1+g_2}(A)x\leq z_1+\ldots+z_{m-\epsilon_2}=m-1$ with $x\in [1,m-1]$. But this implies $mf_1=-mg_1+mg_2=0$ (in view of \eqref{m-unity}), contradicting that $\ord(f_1)=2m$, which completes CASE 4.2.2.

\smallskip
\noindent
CASE 4.2.3: \ $f_1=g_1$

In this case, $\vp_{f_1}(A)=\vp_{g_1}(A)=m-1$ and $\vp_{f_1}(B)=\vp_{g_1}(C)=0$. In view of \eqref{f_2inSupp}, we may w.l.o.g. assume $f_2\in \supp(A)$. We have $f_2\neq f_1=g_1$ while we can assume $f_2\neq g_2$ else CASE 4.2.1 completes the proof. Therefore  \be\label{wefillt}f_2=-xg_1+g_2=-xf_1+g_2\quad \mbox{ for some $x\in [1,m-1]$}.\ee
 Likewise, if $g_2\in \supp(A)$, then $g_2=-yf_1+f_2$ for some $y\in [1,m-1]$, implying $$f_2=-xf_1+g_2=-xf_1-yf_1+f_2,$$ in which case $(x+y)f_1=0$ with $x+y\in [2, 2m-2]$, contradicting that $\ord(f_1)=2m$. Therefore we conclude that $g_2\notin \supp(A)$. As a result, all elements in $\supp(A)\setminus \{g_1\}$ have the form  $-z_ig_1+g_2=-z_if_1+g_2$ with $z_i\in [1,m-1]$.

 Let $-zg_1+g_2\in \supp(A)\setminus\{g_1\}$ be arbitrary. Let us show that $z\geq x$. If $z=x$, this is trivial, so suppose $z\neq x$.  Then $-yf_1+f_2=-zg_1+g_2=-zf_1+g_2$ for some $y\in [1,m-1]$. In this case, \eqref{wefillt} implies $$-yf_1+f_2=-zf_1+g_2=-zf_1+(xf_1+f_2),$$ yielding $(x-z+y)f_1=0$. Consequently, since $\ord(f_1)=2m$ with $x-z+y\in [-(m-1)+2,2(m-1)-1]=[-m+3,2m-3]$, we see that $z=x+y\geq x+1$, as claimed.

 All terms of $A$ not equal to $g_1=f_1$ have the form $-z_ig_1+g_2$. There are at least $|A|-\vp_{g_1}(A)=|A|-m+1\geq \frac{m}{2}$ such terms all with $z_i\geq x$ as shown above. Consequently,  $\frac{m}{2}x\leq z_1+\ldots+z_{m-\epsilon_2}=m-1$, which implies $x=1$. Hence
 $$f_2=-xf_1+g_2=-f_1+g_2,$$ which yields $mf_2=-mf_1+mg_2=0$ (in view of \eqref{m-unity}), contradicting that $\ord(f_2)=2m$ and completing the subcase.

\smallskip
\noindent
CASE 4.2.4: \ $\{f_1,\,f_2\}\cap \{g_1,\,g_2\}=\emptyset$.

Let $$a_i=\vp_{-if_1+f_2}(A)\quad\und\quad b_i=\vp_{-ig_1+g_2}(A)\quad\mbox{ for $i\in [1,m-1]$}.$$
Let $$c=\left|\gcd\left(\prod_{i=1}^{m-\epsilon_2}(-z_ig_1+g_2),\,
\prod_{i=1}^{m-\epsilon_1}(-y_if_1+f_2)\right)\right|.$$ Thus $c$ counts the number of terms of $A$ simultaneously equal to some $-y_if_1+f_2$ as well as some $-z_jg_1+g_2$. In view of the hypothesis $\{f_1,\,f_2\}\cap \{g_1,\,g_2\}=\emptyset$, we see that every term of $A$ is either equal to some $-y_if_1+f_2$ or to some $-z_jg_1+g_2$. As a result, the inclusion-exclusion principle gives \be\label{getitl}\Sum{i=1}{m-1}a_i+\Sum{i=1}{m-1}b_i-c=|A|\geq \frac32 m-1.\ee
Note $-f_1+f_2$ and $-g_1+g_2$ have order $m$ (in view of \eqref{m-unity}), meaning $\{-f_1+f_2,\,-g_1+g_2\}\cap \{f_1,\,f_2,\,g_1,\,g_2\}=\emptyset$. Consequently, if $-f_1+f_2$ occurs in $A$, then it must be equal to some $-zg_1+g_2$ with $z\in [1,m-1]$. It follows that $c\geq a_1$. Likewise, if $-g_1+g_2$ occurs in $A$, then it must be equal to some $-yf_1+f_2$, so that $c\geq b_1$. Averaging these estimates, we obtain $$c\geq \frac{a_1+b_1}{2}.$$ Applying this estimate in \eqref{getitl} along with the pigeon-hole principle, we conclude that either $$\frac12 a_1+\Sum{i=2}{m-1}a_i\geq \frac34m-\frac12\quad\mbox{ or }\quad \frac12 b_1+\Sum{i=2}{m-1}b_i\geq \frac34m-\frac12,$$ and we w.l.o.g. assume the former: \be\label{run1}\frac12 a_1+\Sum{i=2}{m-1}a_i\geq \frac34m-\frac12.\ee
By definition of the $a_i$, we have \be\label{run2}a_1+\Sum{i=2}{m-1}2a_i\leq a_1+2a_2+3a_3+\ldots+(m-1)a_1\leq y_1+\ldots+y_{m-\epsilon_1}=m-1.\ee Combining \eqref{run1} and \eqref{run2} yields $$\frac32m-1\leq 2\left(\frac{a_1}{2}+\Sum{i=2}{m-1}a_i\right)\leq m-1,$$ which is a contradiction, concluding CASE 4.

\bigskip

If $|U_3|=\mathsf D(G)$, then it possible to also apply Main Proposition \ref{5.4} to $U_3$ and (by symmetry) re-index  the $U_i$ with $i\in [1,3]$ in any fashion. Consequently, if one of $U_1$, $U_2$ or $U_3$ has the same type from among I(a), I(b) and II, then we may w.l.o.g. re-index the $U_i$ so that $U_1$ and $-U_2$ have the same type and apply CASE 1, 2 or 4 to  yield the desired conclusion (note $U_i$ and $-U_i$ have the same type). On the other hand, if $U_1$, $U_2$ and $U_3$ have distinct types I(a), I(b) and II, then we my re-index the $U_i$ so that $U_1$ has type I(b) and $-U_2$ has type I(a), in which case CASE 3 completes the proof. In summary, the proof is now complete when $|U_3|=\mathsf D(G)$, so we instead assume $$|U_3|=\mathsf D(G)-1.$$ By Assertion A, this is only possible if $$|V_1|=2, \quad  |A|=\frac{\mathsf D(G)+1}{2}\quad\und\quad |B|=|C|=
\frac{\mathsf D(G)-1}{2}\quad\mbox{ with\quad $\mathsf D(G)=mn+m-1$  odd},$$ which we now assume for the final two cases of the proof, where by symmetry we now assume $-U_2$ has type II.

\smallskip
\noindent
CASE 5: \ $U_1$ is of type I(b) and $-U_2$ is of type II, say
\[
U_1 = e_1^{m-1} \prod_{i=1}^{mn} (x_{i}e_1+e_2) \quad \und\quad
-U_2 = f_1^{sm-1}f_2^{(n-s)m+\epsilon} \prod_{i=1}^{m-\epsilon}
      (-y_{i}f_1+f_2),
\] where $\{e_1,\,e_2\}$ is a basis for $G$ with $\ord(e_2)=mn>m$ and $\ord(e_1)=m$, where $\{f_1,\,f_2\}$ is a generating set for $G$ with $\ord(f_2)=mn$ and $\ord(f_1)\geq 2m$, and where $y_1+\ldots+y_{m-\epsilon}=m-1$ with $y_i\in [1,m-1]$, $\epsilon\in [1,m-1]$ and $s\in [1,n-1]$. Moreover,  either $s=1$ or $mf_1=mf_2$, with both holding when $n=2$.

\medskip

Since $|A|\geq \frac32 m>m-1$, we must have $f_\nu\in \supp(A)$ for some $\nu\in [1,2]$. Since
$\ord(f_\nu)\geq 2m >m=\ord(e_1)$, we cannot have $f_\nu=e_1$. Thus $f_\nu\in \la e_1\ra +e_2$. It is easily noted that any $g\in \la e_1\ra +e_2$ has $\ord(g)=\ord(e_2)=mn$. Moreover, $U_1$ will also have type I(b) using the basis $\{e_1,\,g\}$ replacing each $x_i$ with $x_i-\alpha$, where $g=\alpha e_1+e_2$. Consequently, since $f_\nu\in \la e_1\ra +e_2$ for some $\nu\in [1,2]$, we see that we may w.l.o.g. assume
\be\label{static}f_1=e_2\quad\mbox{ or } \quad f_2=e_2.\ee

\smallskip
\noindent
CASE 5.1: \ $n\geq 3$

Let us first show that \be\label{mickey4}mf_2=me_2.\ee If $s>1$, then this follows from Main Proposition \ref{5.4} and \eqref{static}. If  $s=1$, then $|A|=\frac{mn+m}{2}\geq 2m>2m-2$ (in view of $n\geq 3$), whence $f_2\in \supp(A)$.
Hence, by the argument above CASE 5.1, we may w.l.o.g. assume $f_2=e_2,$ implying  $mf_2=me_2$ in this case as well. Thus \eqref{mickey4} is established.

\medskip

Let $H=\la me_2\ra$. Then $G/H\cong C_m^2$.
Since $n\geq 3$, we have $|C|=|B|=\frac{mn+m-2}{2}\geq 2m-1=\mathsf D(G/H)$.
Let $B'\mid B$ be a subsequence with $|B'|=2m-1$ and let $B'=e_1^t\cdot b_1\cdot\ldots\cdot b_{2m-1-t}$, where $\vp_{e_1}(B')=t\in [0,m-1]$. Then we may w.l.o.g. assume $b_i=x_ie_1+e_2$ for $i\in [1,2m-1-t]$. Since $2m-1-t\leq 2m-1$ and since $t\leq m-1$, it is readily seen that the only way $B'$ can contain a nontrivial  subsequence $T\mid B'$ with $\sigma(T)\in H=\la me_2\ra$ is if $T$ contains precisely $m$ terms from $b_1\cdot\ldots\cdot b_{2m-1-t}$, in which case $\sigma(T)=me_2$. Consequently, since $|B'|=2m-1=\mathsf D(G/H)$, we conclude that there exits a subsequence $T\mid B'$ with $$\sigma(T)=me_2=mf_2\quad\und\quad |T|\leq m+t\leq 2m-1.$$ Moreover, $T$ will be a proper subsequence of $B$ unless $n=3$ (so that $2m-1=|B|=|B'|=|T|$) and (w.l.o.g. re-indexing the $x_ie_1+e_2$) $$B=e_1^{m-1}\prod_{i=1}^m (x_ie_1+e_2)\quad\mbox{ with  } \quad \Sum{i=1}{m}x_i\equiv 1\mod m.$$
Since $|C|\geq 2m-1$, Lemma \ref{lem-tech} ensures that there is a subsequence $R\mid C$ with $\sigma(R)=mf_2$ and $|R|\leq 2m-1$. Moreover, $R$ will be a proper subsequence unless $n=3$ and $$C=f_1^{m-1}f_2^\epsilon\prod_{i=1}^{m-\epsilon}(-y_if_1+f_2).$$ Now $(-T)R$ is a nontrivial zero-sum subsequence of $(-B)C=U_3$. Since $U_3$ is an atom, this is only possible if $T=B$ and $R=C$. Thus $n=3$ and
$$U_3=(-e_1)^{m-1}\prod_{i=1}^m (-x_ie_1-e_2)f_1^{m-1}f_2^\epsilon\prod_{i=1}^{m-\epsilon}(-y_if_1+f_2),$$ where $\Sum{i=1}{m}x_i\equiv 1\mod m$ and $\Sum{i=1}{m-\epsilon}y_i=m-1$. Since $\vp_{f_2}(-U_2)=(n-s)m+\epsilon\geq m+\epsilon>\epsilon$, we conclude that $f_2\in \supp(A)$, whence (as argued before CASE 5.1) we may w.l.o.g. assume $e_2=f_2.$ As a result, we see that $(-x_1e_1-e_2)(-y_1f_1+e_2)f_1^{y_1}(-e_1)^z$, where $z\in [0,m-1]$ is the integer such that $z+x_1\equiv 0\mod m$, will be a zero-sum subsequence of $U_3$ of length $2+y_1+z\leq 2m<4m-2=|U_3|$, contradicting that $U_3$ is an atom.

\smallskip
\noindent
CASE 5.2: \ $n=2$.

Similar to CASE 4.2, we now have $\ord(e_2)=\ord(f_2)=\ord(f_1)=2m$, \ $s=1$, \ $\epsilon\geq 2$, \ $m\geq 4$ even (since $\mathsf D(G)=3m-1$ is odd),  and
\be\label{new-m-unity}mf_1=mf_2=me_2,\ee with
$$U_1=e_1^{m-1}\prod_{i=1}^{2m}(x_{i}e_1+e_2) \quad \und \quad  -U_2=f_1^{m-1}f_2^{m+\epsilon_2}\prod_{i=1}^{m-\epsilon}(-y_{i}f_1+f_2).
 $$
We handle several subcases.

\smallskip
\noindent
CASE 5.2.1: \ $e_1=-f_1+f_2$.

Let $t$ be the number of terms from $C$ of the form $-yf_1+f_2$ with $y\in [1,m-1]$. Then, since $e_1=-f_1+f_2$, we see that $\vp_{e_1}(A)\leq m-\epsilon-t$, so that \be\label{wellt}\vp_{-e_1}(-B)\geq \epsilon-1+t.\ee
By \eqref{static}, we have  $f_1=e_2$ or $f_2=e_2$. In either case, the hypothesis $e_1=-f_1+f_2$ ensures that $f_1,\,f_2\in \la e_1\ra+e_2$. Thus there are $|C|-t=\frac32m-1-t$ terms of $C$ from $\la e_1\ra+e_2$, say $c_1\cdot\ldots\cdot c_{\ell_1}\mid C$ with $c_i\in \la e_1\ra+e_2$ and $\ell_1=\frac32m-1-t$, and there are (by \eqref{wellt}) $$\ell_2:=|B|-\vp_{-e_1}(-B)\leq \frac32m-1-(\epsilon-1+t)=\frac32m-\epsilon-t\leq \frac32m-2-t$$ terms of $-B$ from $\la e_1\ra-e_2$, say $b_1\cdot \ldots\cdot b_{\ell_2}\mid -B$ with $b_i\in \la e_1\ra-e_2$ and $\ell_2<\ell_1$. Consequently, $(-e_1)^{\vp_{-e_1}(-B)}(b_1+c_1)\cdot\ldots\cdot (b_{\ell_2}+c_{\ell_2})\in \Fc(\la e_1\ra)$ is a sequence of terms from $\la e_1\ra\cong C_m$ of length $|B|=\frac32m-1\geq m=\mathsf D(\la e_1\ra)$. Thus the proper (in  view of $\ell_1>\ell_2$) subsequence $(-e_1)^{\vp_{-e_1}(-B)}b_1\cdot\ldots\cdot b_{\ell_2}\cdot c_1\cdot\ldots\cdot c_{\ell_2}$ of $(-B)C=U_3$ contains a nontrivial zero-sum subsequence, contradicting that $U_3$ is an atom.

\smallskip
\noindent
CASE 5.2.2: \ $f_2\in \supp(A)$.

In this case, we may assume $$f_2=e_2$$ per the argument before CASE 5.1.

First suppose that $e_1\notin \supp(A)$. Then $\vp_{-e_1}(-B)=m-1$. Since $|B|=\frac32m-1>m-1$, there must be some $-xe_1-e_2\in \supp(-B)$. If $\vp_{e_2}(C)=\vp_{f_2}(C)>0$, then $(-xe_1-e_2)e_2(-e_1)^z$, where $z\in [0,m-1]$ is the integer with $z+x\equiv 0\mod m$, will be a zero-sum subsequence of $U_3$ of length $z+2\leq m+1<3m-2=|U_3|$, contradicting that $U_3$ is an atom. Therefore we instead assume $\vp_{f_2}(C)=0$. Thus $m-1\geq \vp_{f_1}(C)\geq |C|-(m-\epsilon)=\frac{m}{2}-1+\epsilon\geq \frac{m}{2}$, implying $\epsilon\leq \frac{m}{2}$, and there are at least $|C|-m+1\geq \frac{m}{2}>0$ terms in $C$ of the form $-y_if_1+e_2$ with $y_i\in [1,\epsilon]\subseteq [1,\frac{m}{2}]$. Let $-yf_1+e_2\in\supp(C)$ with $y\in [1,\frac{m}{2}]$ be one such term. Then $f_1^{y}(-yf_1+e_2)(-xe_1-e_2)(-e_1)^z$, where $z\in [0,m-1]$ is the integer such that $x+z\equiv 0\mod 0$, is a zero-sum subsequence of $U_3$ of length $y+z+2\leq\frac32m+1<3m-2=|U_3|$, contradicting that $U_3$ is an atom. So we instead conclude that $e_1\in \supp(A)$. As a result, since $\ord(e_1)=m<2m=\ord(f_1)=\ord(f_2)=\ord(e_2)$, we must have
\be\label{samei}e_1=-yf_1+f_2=-yf_1+e_2\quad\mbox{ for some $y\in [1,\epsilon]$}.\ee Furthermore, since $me_1=0$, we conclude from $\ord(e_2)=2m$ that $y$ is odd, and in view of CASE 5.2.1, we can assume $y\geq 3$.

\bigskip

Suppose $f_1\in \supp(A)$. Then $f_1=xe_1+e_2$ for some $x\in \Z$. Combining this with \eqref{samei} yields $-e_1+e_2=yf_1=xye_1+ye_2$, which implies $y\equiv 1\mod 2m$. Hence, since $y\in [1,m-1]$, we conclude that $y=1$, which is contrary to our above assumption. So we may instead assume $f_1\notin \supp(A)$, implying $$\vp_{f_1}(C)=m-1.$$

\bigskip

Each term of $A$ equal to $e_1=-yf_1+f_2=-yf_1+e_2$  is also equal to some $-y_if_1+f_2$. Thus $3\vp_{e_1}(A)\leq \vp_{e_1}(A)y\leq y_1+\ldots+y_{m-\epsilon}=m-1$, implying $\vp_{e_1}(A)\leq \frac{m-1}{3}$ and $$\vp_{-e_1}(-B)\geq \frac{2m-2}{3}\geq\frac{m}{2}.$$ Since $\vp_{f_1}(C)=m-1$, we find that there are precisely $|C|-m+1=\frac{m}{2}$ terms of $C$ either equal to $e_2=f_2$ or $-y_if_1+e_2$ for some $i$. Hence, since $y_1+\ldots+y_{m-\epsilon}=m-1=\vp_{f_1}(C)$ with the $y_i\in [1,m-1]$, we see that we can find disjoint subsequences $T_1\cdot\ldots\cdot T_{m/2}\mid C$ with each $T_i\in \Fc(G)$ a  subsequence having $\sigma(T_i)=e_2$. There are at least $|B|-m+1=\frac{m}{2}$ terms of $-B$ of the form $-xe_1-e_2$, say $b_1\cdot\ldots\cdot b_{m/2}\mid -B$ with $b_i\in \la e_1\ra-e_2$ for all $i$. Now $(\sigma(T_1)+b_1)\cdot\ldots \cdot (\sigma(T_{m/2})+b_{m/2})(-e_1)^{m/2}\in \Fc(\la e_1\ra)$ is a subsequence of terms from $\la e_1\ra$ of length $m=\mathsf D(\la e_1\ra)$. Consequently, the subsequence $T_1\cdot\ldots\cdot T_{m/2} \cdot b_1\cdot\ldots\cdot b_{m/2}\cdot (-e_1)^{m/2}$ of $(-B)C=U_3$ contains a nontrivial zero-sum subsequence of length at most $|T_1|+\ldots+|T_{m/2}|+m\leq |C|+m=\frac52m-1<3m-2=|U_3|$, contradicting that $U_3$ is an atom and completing CASE 5.2.1.

\smallskip
\noindent
CASE 5.2.3: \ $f_2\notin \supp(A)$.

Since $f_2\notin\supp(A)$, all terms of $A$ not equal to $f_1$ are equal to some $-y_if_1+f_2$, and  there are at least $|A|-m+1=\frac{m}{2}+1$ such terms of $A$. If $y_i\geq 2$ for all these terms, then we obtain the contradiction $m+2=(\frac{m}{2}+1)2\leq y_1+\ldots+y_{m-\epsilon}=m-1$. Thus  $-f_1+f_2\in \supp(A)$. If $-f_1+f_2=xe_1+e_2$ for some $x\in \Z$, then \eqref{new-m-unity} implies $0=-mf_1+mf_2=xme_1+me_2=me_2$, contradicting that $\ord(e_2)=2m$.  Therefore we instead conclude that $-f_1+f_2=e_1,$ so that CASE 5.2.1 completes the proof of CASE 5.

\smallskip
\noindent
CASE 6: \ $U_1$ is of type I(a) and $-U_2$ is of type II, say
\[
U_1 = e_1^{mn-1} \prod_{i=1}^{m} (x_{i}e_1+e_2) \quad \und\quad
-U_2 = f_1^{sm-1}f_2^{(n-s)m+\epsilon} \prod_{i=1}^{m-\epsilon}
      (-y_{i}f_1+f_2),
\] where $\{e_1,\,e_2\}$ is a basis for $G$ with $\ord(e_2)=m$ and $\ord(e_1)=mn>m$, where $\{f_1,\,f_2\}$ is a generating set for $G$ with $\ord(f_2)=mn$ and $\ord(f_1)\geq 2m$, and where $y_1+\ldots+y_{m-\epsilon}=m-1$ with $y_i\in [1,m-1]$, $\epsilon\in [1,m-1]$ and $s\in [1,n-1]$. Moreover,  either $s=1$ or $mf_1=mf_2$, with both holding when $n=2$.

\smallskip
\noindent
CASE 6.1: \ $n\geq 3$.

Since $|A|=\frac{mn+m}{2}>m$, we must have $e_1\in \supp(A)$. If  $e_1=-y_if_1+f_2$ for some $y_i\in [1,m-1]$, then we obtain the contradiction $|A|\leq \vp_{e_1}(A)+m\leq (m-\epsilon)+m\leq  2m-1<\frac{mn+m}{2}=|A|$ (in view of $n\geq 3$). Therefore either \be\label{stand}e_1=f_1\quad\mbox{ or } \quad e_1=f_2.\ee

 Suppose $s=1$. If $e_1=f_2$, then $\frac{mn+m}{2}=|A|\geq\vp_{e_1}(A)=(n-1)m+\epsilon\geq mn-m+1$, contradicting that $n\geq 3$. If $e_1=f_1$, then $|A|\leq \vp_{e_1}(A)+m\leq \vp_{f_1}(-U_2)+m= 2m-1<\frac{mn+m}{2}=|A|$, again in view of $n\geq 3$, which is  a contradiction. So (in view of \eqref{stand}) we may instead assume $s>1$, whence \be\label{goofy}mf_2=mf_1=me_1,\ee where the first equality follows from Main Proposition \ref{5.4} and the second from \eqref{stand}.

The argument is now similar to  CASE 5.1. Let $H=\la me_1\ra$. Then $G/H\cong C_m^2$.
Since $n\geq 3$, we have $|C|=|B|=\frac{mn+m-2}{2}\geq 2m-1=\mathsf D(G/H)$.
Let $B'\mid B$ be a subsequence with $|B'|=2m-1$. If $\vp_{e_1}(B')\geq m$, then $B'$ will contain a subsequence $T=e_1^{m}$ with $\sigma(T)=me_1$. If $\vp_{e_1}(B')<m$, then this is only possible if $$B'=e_1^{m-1}\prod_{i=1}^m (x_ie_1+e_2)\quad\mbox{ with  } \quad \Sum{i=1}{m}x_i\equiv 1\mod mn,$$ in which case it is easily seen that $T=B'$ is a subsequence of $B'$ with $\sigma(T)=me_1$.
Since $|C|\geq 2m-1$, Lemma \ref{lem-tech} ensures that there is a subsequence $R\mid C$ with $\sigma(R)=mf_2$ and $|R|\leq 2m-1$. Moreover, $R$ will be a proper subsequence unless $n=3$ and $$C=f_1^{m-1}f_2^\epsilon\prod_{i=1}^{m-\epsilon}(-y_if_1+f_2).$$ Now $(-T)R$ is a nontrivial zero-sum subsequence of $(-B)C=U_3$ (in view of \eqref{goofy}). Since $U_3$ is an atom, this is only possible if $T=B'=B$ and $R=C$. Thus $n=3$ and
$$U_3=(-e_1)^{m-1}\prod_{i=1}^m (-x_ie_1-e_2)f_1^{m-1}f_2^\epsilon\prod_{i=1}^{m-\epsilon}(-y_if_1+f_2).$$  As a result, since $\vp_{f_2}(-U_2)=(n-s)m+\epsilon\geq m+\epsilon>m>\epsilon$, we conclude that $f_2\in \supp(A)$. Moreover, $\vp_{f_2}(A)\geq m+\epsilon-\vp_{f_2}(C)=m$. Hence, as the only term in $U_1$ with multiplicity at least $m$ is $e_1$ (recall $|\supp(U_1)|\geq 3$ as remarked after Main Proposition \ref{5.4}), we conclude that $e_1=f_2$, in which case $(-e_1)(f_2)=(-e_1)(e_1)$ is a proper zero-sum subsequence of $U_3$, contradicting that $U_3$ is an atom.

\smallskip
\noindent
CASE 6.2: \ $n=2$.

Similar to CASE 5.2, we now have $\ord(e_1)=\ord(f_2)=\ord(f_1)=2m$, \ $s=1$, \ $\epsilon\geq 2$, \ $m\geq 4$ even (since $\mathsf D(G)=3m-1$ is odd),  and
\be\label{newer-m-unity}mf_1=mf_2=me_1,\ee with
$$U_1=e_1^{2m-1}\prod_{i=1}^{m}(x_{i}e_1+e_2) \quad \und \quad  -U_2=f_1^{m-1}f_2^{m+\epsilon}\prod_{i=1}^{m-\epsilon}(-y_{i}f_1+f_2).
 $$
Since $|A|=\frac32m>m$, we must have $e_1\in \supp(A)$. We have three possibilities for $e_1$.

\medskip

Suppose $e_1=f_1$. Then $\vp_{e_1}(A)=\vp_{f_1}(A)=m-1$, implying $$\vp_{-e_1}(-B)=m$$ and $\vp_{f_1}(C)=0$. Let $T=c_1\cdot\ldots\cdot c_m\mid C$ be any length $m$ subsequence of $C$. As $\vp_{f_1}(C)=0$, each $c_i=-z_if_1+f_2=-z_ie_1+f_2$ for some $z_i\in [0,m-1]$ with $$0\leq z:=z_1+\ldots+z_m\leq y_1+\ldots +y_{m-\epsilon}=m-1.$$
Then $\sigma(T)=-zf_1+mf_2=(m-z)e_1$ (in view of \eqref{newer-m-unity} and $e_1=f_1$), in which case $(-e_1)^{m-z}T$ is a zero-sum subsequence of $(-B)C=U_3$ of length $m-z+|T|\leq 2m<3m-2=|U_3|$, contradicting that $U_3$ is an atom.

\medskip

Suppose $e_1=f_2$. Then $\vp_{e_1}(A)=\vp_{f_2}(A)=m+\epsilon\leq |A|=\frac32 m$, implying $\epsilon\leq \frac{m}{2}$,  $$\vp_{-e_1}(-B)=m-1-\epsilon\geq \frac{m}{2}-1>0$$ and $\vp_{f_2}(C)=0$. Since $\vp_{f_2}(A)=m+\epsilon$, it follows that there are at most $|A|-\vp_{f_2}(C)=\frac{m}{2}-\epsilon$ terms of $A$ of the form $-y_if_1+f_2=-y_if_1+e_1$, meaning there are at least $m-\epsilon-(\frac{m}{2}-\epsilon)=\frac{m}{2}$ terms of $C$ of this form, say $b_1\cdot\ldots\cdot b_\ell\mid C$ with w.l.o.g. $b_i=-y_if_1+f_2=-y_if_1+e_1$ for $i\in [1,\ell]$ and $\ell\geq \frac{m}{2}$. If $b_i\geq 2$ for all $i\in [1,\ell]$, then we obtain the contradiction $m\leq 2\ell\leq b_1+\ldots+b_\ell\leq y_1+\ldots+y_{m-\epsilon}=m-1$. Therefore we may assume
$y_i=1$ for some $i\in [1,\ell]$, meaning $$-f_1+e_1\in \supp(C).$$ Since $\vp_{f_2}(A)=m+\epsilon$, there are also at most $|A|-m-\epsilon=\frac{m}{2}-\epsilon$ terms of $A$ equal to $f_1$, whence $\vp_{f_1}(C)\geq m-1-(\frac{m}{2}-\epsilon)=\frac{m}{2}-1+\epsilon>0$. Hence $f_1(-e_1)(-f_1+e_1)$ is a zero-sum subsequence of $(-B)C=U_3$ of length $3<3m-2=|U_3|$, contradicting that $U_3$ is an atom.

\medskip

It remains to consider the case when $e_1=-yf_1+f_2$ for some $y\in [1,m-1]$. Moreover, in view of \eqref{newer-m-unity} and $\ord(e_1)=2m$, we must have $y$ even, whence $y\geq 2$. Thus $2\vp_{e_1}(A)\leq y_1+\ldots+y_{m-\epsilon}=m-1$, implying $\vp_{e_1}(A)\leq \frac{m-1}{2}.$
But now $\frac32 m=|A|\leq \vp_{e_1}(A)+m\leq\frac{m-1}{2}+m$, which is a proof concluding contradiction.
\end{proof}

\providecommand{\bysame}{\leavevmode\hbox to3em{\hrulefill}\thinspace}
\providecommand{\MR}{\relax\ifhmode\unskip\space\fi MR }
\providecommand{\MRhref}[2]{%
  \href{http://www.ams.org/mathscinet-getitem?mr=#1}{#2}
}
\providecommand{\href}[2]{#2}


\begin{thebibliography}{10}

\bibitem{An97a}
D.F. Anderson, \emph{Elasticity of factorizations in integral domains{\rm \,:}
  a survey}, Factorization in {I}ntegral {D}omains, Lect. Notes Pure Appl.
  Math., vol. 189, Marcel Dekker, 1997, pp.~1 -- 29.

\bibitem{A-C-C-S95}
D.F. Anderson, P.-J. Cahen, S.T. Chapman, and W.W. Smith, \emph{Some
  factorizations properties of the ring of integer-valued polynomials},
  Zero-{D}imensional {C}ommutative {R}ings, Lect. Notes Pure Appl. Math., vol.
  171, Marcel Dekker, 1995, pp.~95 -- 113.

\bibitem{Ba-Ge14b}
N.R. Baeth and A.~Geroldinger, \emph{Monoids of modules and arithmetic of
  direct-sum decompositions}, Pacific J. Math. \textbf{271} (2014), 257 -- 319.

\bibitem{Ba-Sm15}
N.R. Baeth and D.~Smertnig, \emph{Factorization theory from commutative to
  noncommutative settings}, J. Algebra, to appear.

\bibitem{Ba-Wi13a}
N.R. Baeth and R.~Wiegand, \emph{Factorization theory and decomposition of
  modules}, Am. Math. Mon. \textbf{120} (2013), 3 -- 34.

\bibitem{B-C-H-P08}
P.~Baginski, S.T. Chapman, N.~Hine, and J.~Paixao, \emph{On the asymptotic
  behavior of unions of sets of lengths in atomic monoids}, Involve, a journal
  of mathematics \textbf{1} (2008), 101 -- 110.

\bibitem{B-G-G-P13a}
P.~Baginski, A.~Geroldinger, D.J. Grynkiewicz, and A.~Philipp, \emph{Products
  of two atoms in {K}rull monoids and arithmetical characterizations of class
  groups}, Eur. J. Comb. \textbf{34} (2013), 1244 -- 1268.

\bibitem{B-C-C-M09}
M.~Banister, J.~Chaika, S.T. Chapman, and W.~Meyerson, \emph{A theorem on
  accepted elasticity in certain local arithmetical congruence monoids}, Abh.
  Math. Semin. Univ. Hamb. \textbf{79} (2009), 79 -- 86.

\bibitem{Bh-Ha-SP10b}
G.~Bhowmik, I.~Halupczok, and J.-C. Schlage-Puchta, \emph{The structure of
  maximal zero-sum free sequences}, Acta Arith. \textbf{143} (2010), 21 -- 50.

\bibitem{Bl-Ga-Ge11a}
V.~Blanco, P.~A. Garc{\'i}a-S{\'a}nchez, and A.~Geroldinger,
  \emph{Semigroup-theoretical characterizations of arithmetical invariants with
  applications to numerical monoids and {K}rull monoids}, Illinois J. Math.
  \textbf{55} (2011), 1385 -- 1414.

\bibitem{Ca-Ch95}
P.-J. Cahen and J.-L. Chabert, \emph{Elasticity for integer-valued
  polynomials}, J. Pure Appl. Algebra \textbf{103} (1995), 303 -- 311.

\bibitem{Ch-Gl00a}
S.T. Chapman and S.~Glaz, \emph{One hundred problems in commutative ring
  theory}, Non-{N}oetherian {C}ommutative {R}ing {T}heory, Mathematics and
  {I}ts {A}pplications, vol. 520, Kluwer {A}cademic {P}ublishers, 2000, pp.~459
  -- 476.

\bibitem{Ch-Cl05}
S.T. Chapman and B.~McClain, \emph{Irreducible polynomials and full elasticity
  in rings of integer-valued polynomials}, J. Algebra \textbf{293} (2005), 595
  -- 610.

\bibitem{Ch-Sm90a}
S.T. Chapman and W.W. Smith, \emph{Factorization in {D}edekind domains with
  finite class group}, Isr. J. Math. \textbf{71} (1990), 65 -- 95.

\bibitem{Sa-Ch14a}
F.~Chen and S.~Savchev, \emph{Long minimal zero-sum sequences in the groups
  ${C}_2^{r-1} \oplus {C}_{2k}$}, Integers \textbf{14} (2014), Paper A23.

\bibitem{Cz81}
A.~Czogala, \emph{Arithmetic characterization of algebraic number fields with
  small class number}, Math. Z. \textbf{176} (1981), 247 -- 253.

\bibitem{Fa06a}
A.~Facchini, \emph{Krull monoids and their application in module theory},
  Algebras, {R}ings and their {R}epresentations (A.~Facchini, K.~Fuller, C.~M.
  Ringel, and C.~Santa-Clara, eds.), World Scientific, 2006, pp.~53 -- 71.

\bibitem{Fr-Ge08}
M.~Freeze and A.~Geroldinger, \emph{Unions of sets of lengths}, Funct.
  Approximatio, Comment. Math. \textbf{39} (2008), 149 -- 162.

\bibitem{Ga-Ge03b}
W.~Gao and A.~Geroldinger, \emph{On zero-sum sequences in $\mathbb{Z} /n
  \mathbb{Z} \oplus \mathbb{Z} /n \mathbb{Z}$}, Integers \textbf{3} (2003),
  Paper A08, 45p.

\bibitem{Ga-Ge09b}
\bysame, \emph{On products of $k$ atoms}, Monatsh. Math. \textbf{156} (2009),
  141 -- 157.

\bibitem{Ga-Ge-Gr10a}
W.~Gao, A.~Geroldinger, and D.J. Grynkiewicz, \emph{Inverse zero-sum problems
  {III}}, Acta Arith. \textbf{141} (2010), 103 -- 152.

\bibitem{Ge09a}
A.~Geroldinger, \emph{Additive group theory and non-unique factorizations},
  Combinatorial {N}umber {T}heory and {A}dditive {G}roup {T}heory
  (A.~Geroldinger and I.~Ruzsa, eds.), Advanced Courses in Mathematics CRM
  Barcelona, Birkh{\"a}user, 2009, pp.~1 -- 86.

\bibitem{Ge-HK06a}
A.~Geroldinger and F.~Halter-Koch, \emph{Non-{U}nique {F}actorizations.
  {A}lgebraic, {C}ombinatorial and {A}nalytic {T}heory}, Pure and Applied
  Mathematics, vol. 278, Chapman \& Hall/CRC, 2006.

\bibitem{Ge-Ka-Re15a}
A.~Geroldinger, F.~Kainrath, and A.~Reinhart, \emph{Arithmetic of seminormal
  weakly {K}rull monoids and domains}, J. Algebra, to appear.

\bibitem{Ge-Li-Ph12}
A.~Geroldinger, M.~Liebmann, and A.~Philipp, \emph{On the {D}avenport constant
  and on the structure of extremal sequences}, Period. Math. Hung. \textbf{64}
  (2012), 213 -- 225.

\bibitem{Ge-Ru09}
A.~Geroldinger and I.~Ruzsa, \emph{Combinatorial {N}umber {T}heory and
  {A}dditive {G}roup {T}heory}, Advanced Courses in Mathematics - CRM
  Barcelona, Birkh{\"a}user, 2009.

\bibitem{Ge-Sc16a}
A.~Geroldinger and Wolfgang~A. Schmid, \emph{A characterization of class groups
  via sets of lengths}, submitted.

\bibitem{Ge-Sc92}
A.~Geroldinger and R.~Schneider, \emph{On {D}avenport's constant}, J. Comb.
  Theory, Ser. A \textbf{61} (1992), 147 -- 152.

\bibitem{Gi84}
R.~Gilmer, \emph{Commutative {S}emigroup {R}ings}, The University of Chicago
  Press, 1984.

\bibitem{Gr13a}
D.J. Grynkiewicz, \emph{Structural {A}dditive {T}heory}, Developments in
  Mathematics, Springer, 2013.

\bibitem{HK98}
F.~Halter-Koch, \emph{Ideal {S}ystems. {A}n {I}ntroduction to {M}ultiplicative
  {I}deal {T}heory}, Marcel Dekker, 1998.

\bibitem{Ka05a}
F.~Kainrath, \emph{Elasticity of finitely generated domains}, Houston J. Math.
  \textbf{31} (2005), 43 -- 64.

\bibitem{Ki01}
H.~Kim, \emph{The distribution of prime divisors in {K}rull monoid domains}, J.
  Pure Appl. Algebra \textbf{155} (2001), 203 -- 210.

\bibitem{Ki-Pa01}
H.~Kim and Y.~S. Park, \emph{{K}rull domains of generalized power series}, J.
  Algebra \textbf{237} (2001), 292 -- 301.


\bibitem{Re10c}
C.~Reiher, \emph{A proof of the theorem according to which every prime number
  possesses property {B}}, PhD Thesis, Rostock, 2010 (2010).

\bibitem{Sa-Za82}
L.~Salce and P.~Zanardo, \emph{Arithmetical characterization of rings of
  algebraic integers with cyclic ideal class group}, Boll. Unione. Mat. Ital.,
  VI. Ser., D, Algebra Geom. \textbf{1} (1982), 117 -- 122.

\bibitem{Sc09c}
W.A. Schmid, \emph{Arithmetical characterization of class groups of the form
  $\mathbb{Z} /n \mathbb{Z} \oplus \mathbb{Z} /n \mathbb{Z}$ via the system of
  sets of lengths}, Abh. Math. Semin. Univ. Hamb. \textbf{79} (2009), 25 -- 35.

\bibitem{Sc09b}
\bysame, \emph{Characterization of class groups of {K}rull monoids via their
  systems of sets of lengths{\rm \,:} a status report}, Number {T}heory and
  {A}pplications{\rm \,:} {P}roceedings of the {I}nternational {C}onferences on
  {N}umber {T}heory and {C}ryptography (S.D. Adhikari and B.~Ramakrishnan,
  eds.), Hindustan Book Agency, 2009, pp.~189 -- 212.

\bibitem{Sc10b}
\bysame, \emph{Inverse zero-sum problems {II}}, Acta Arith. \textbf{143}
  (2010), 333 -- 343.

\bibitem{Sm10a}
D.~Smertnig, \emph{On the {D}avenport constant and group algebras}, Colloq.
  Math. \textbf{121} (2010), 179 -- 193.

\bibitem{Sm13a}
\bysame, \emph{Sets of lengths in maximal orders in central simple algebras},
  J. Algebra \textbf{390} (2013), 1 -- 43.

\end{thebibliography}
\end{document}